\documentclass{amsart}
\usepackage{amssymb}
\usepackage{amsmath}

\newtheorem{theorem}{Theorem}[section]
\newtheorem{lemma}[theorem]{Lemma}
\newtheorem{corollary}[theorem]{Corollary}
\newtheorem{example}[theorem]{Example}

\newtheorem{proposition}[theorem]{Proposition}

\newtheorem{remark}[theorem]{Remark}

\numberwithin{equation}{section}

\DeclareMathOperator*{\esssup}{ess\,sup}
\DeclareMathOperator*{\osc}{osc}

\newcommand{ \mr }{ \mathbb{R} }
\newcommand{ \ba }{ \mathbf{a} }
\newcommand{ \m }{ \mathcal{M} }
\newcommand{ \hm }{\widehat{m}}

\def\Xint#1{\mathchoice
    {\XXint\displaystyle\textstyle{#1}}%
     {\XXint\textstyle\scriptstyle{#1}}%
     {\XXint\scriptstyle\scriptscriptstyle{#1}}%
     {\XXint\scriptstyle\scriptscriptstyle{#1}}%
	\!\int}
\def\XXint#1#2#3{{\setbox0=\hbox{$#1{#2#3}{\int}$}
	\vcenter{\hbox{$#2#3$}}\kern-.5\wd0}}

\date\today

\begin{document}

\title[Quasilinear equations with Morrey data]{Global  H\"older continuity of solutions to quasilinear equations with Morrey data}

\author[S.-S. Byun, D.K. Palagachev, P. Shin]{Sun-Sig Byun\and Dian K. Palagachev\and Pilsoo Shin}

\address{Sun-Sig Byun: Seoul National University, Department of Mathematical Sciences and Research Institute of Mathematics, Seoul 08826, Korea}
\email{byun@snu.ac.kr}

\address{Dian K. Palagachev: Politecnico di Bari,
Dipartimento di Meccanica, Matematica e Management,
Via Edoardo Orabona 4, 70125 Bari, Italy}
\email{dian.palagachev@poliba.it}

\address{Pilsoo Shin: Department of Mathematics, Kyonggi University, Suwon 16227, Republic of Korea}
\email{shinpilsoo.math@kgu.ac.kr}



\begin{abstract}
We deal with  general quasilinear divergence-form coercive operators whose   
prototype is the $m$-Lap\-la\-ce\-an operator. The nonlinear terms are given by Carath\'eodory functions and satisfy 
controlled growth  structure conditions with data belonging to suitable Morrey spaces. The fairly non-regular boundary of the underlying domain is supposed to satisfy a capacity density condition which allows domains with exterior corkscrew property.

We prove global boundedness and 
H\"older continuity up to the boundary for the weak solutions of such equations, generalizing this way the classical $L^p$-result of Ladyzhenskaya and Ural'tseva to the settings of the Morrey spaces.
\end{abstract}

\maketitle

\section{Introduction}

The general aim of the present article is to obtain sufficient conditions ensuring boundedness and H\"older continuity up to the boundary for the weak solutions to very general quasilinear elliptic equations with discontinuous ingredients which are controlled within the Morrey functional scales. Precisely, we deal with weak solutions $u\in W^{1,m}_0(\Omega)$ of the  Dirichlet problem
\begin{equation}\label{1}
\begin{cases}
\mathrm{div\,}\big(\ba(x,u,Du)\big)=b(x,u,Du) + \nu & \textrm{in}\ \Omega\\
u=0 & \textrm{on}\ \partial\Omega,
\end{cases} 
\end{equation}
where $\Omega\subset\mr^n,$ $n\geq2,$ is a bounded domain with non-smooth boundary $\partial\Omega$, 
$m\in (1,n],$  $\nu$ is a signed Radon measure with finite total mass $|\nu|(\Omega)<\infty$, and $\ba\colon\Omega\times\mr\times\mr^n\to\mr^n$ and  $b\colon\Omega\times\mr\times\mr^n\to\mr$ are Carath\'eodory maps. Let us stress the reader attention at the very beginning that prototypes of the quasilinear equations studied are these  for the \textit{$m$-Laplace} operator $\mathrm{div\,}\big(|Du|^{m-2}Du\big)$ with $m>1,$ or these for the \textit{$m$-area} type operator $\mathrm{div\,}\left(\big(A+|Du|^2\big)^{\frac{m-2}{2}}Du\right)$ with $m\geq2$ and $A>0.$

Regarding the nonlinear terms in \eqref{1}, we assume \textit{controlled growths} with respect to $u$ and $Du,$ that is,
$$
\begin{cases}
|\ba(x,u,Du)| =\mathcal{O}\left(\varphi(x)+|u|^{\frac{{m^*}(m-1)}{m}}+|Du|^{m-1}\right),\\[4pt]
|b(x,u,Du)| =\mathcal{O}\left(\psi(x)+|u|^{{m^*}-1}+|Du|^{\frac{m({m^*}-1)}{{m^*}}}\right)
\end{cases}
$$ 
as $|z|,|Du|\to\infty,$ and \textit{coercivity} of the differential operator considered
$$
\ba(x,u,Du)\cdot Du\geq\gamma|Du|^m-\Lambda|u|^{{m^*}}-\Lambda\varphi(x)^\frac{m}{m-1}
$$
with non-negative functions $\varphi$ and $\psi$
and  constants $\gamma>0,$  $\Lambda\geq0,$ and  where ${m^*}$ is the Sobolev conjugate of $m.$
In what follows, we will assume that $\varphi$ and $\psi$ are non-negative measurable functions belonging to suitable \textit{Morrey spaces}. 
Namely, we suppose
\begin{equation}\label{2}
\begin{array}{lll}
\varphi\in L^{p,\lambda}(\Omega) &\text{with}\ & p>\frac{m}{m-1},\ \lambda\in(0,n) \ \text{and}\ (m-1)p+\lambda>n \\[4pt]
\psi\in L^{q,\mu}(\Omega), &\text{with} \  &q\geq1,\ \mu\in(0,n)\ \text{and}\ mq+\mu>n.
\end{array}
\end{equation}
A given Radon measure $\nu$ satisfies
\begin{equation}\label{2-1}
\nu\in \mathcal{R}^{1,\omega}(\Omega), \quad \text{with} \quad \omega\in(0,n)\ \text{and}\ m+\omega>n,
\end{equation}
which we will specify later in the next section. It is worth noting that $\varphi \in L^{\frac{m}{m-1}}(\Omega)$, $\psi\in L^{q,\mu}(\Omega)$ (or $\psi \in L^\frac{nm}{nm+m-n}(\Omega)$) and $\nu \in \mathcal{R}^{1,\omega}$, together with the controlled growths
are the \textit{minimal} hypotheses on the data under which the concept of $W^{1,m}_0(\Omega)$-weak solution to \eqref{1}
makes sense.

The non-regular boundary of $\Omega$ will be assumed to satisfy a \textit{density condition} expressed in terms of variational $m$-capacity  (see \eqref{3} below) which requires the complement $\mr^n\setminus\Omega$ to be \textit{uniformly $m$-thick}. This notion is a natural generalization of the \textit{measure density condition,} known also as \textit{(A)-condition} of Ladyzhenskaya and Ural’tseva (cf. \cite{LU1,LU2,LU}), which holds for instance when each point of $\partial\Omega$ supports the \textit{exterior cone property}, excluding this way exterior spikes on $\partial\Omega.$ In that sense, the uniform $m$-thickness condition is satisfied by domains with $C^1$-smooth or Lipschitz continuous boundaries, but it holds also when $\partial\Omega$ is flat in the sense of Reifenberg, including this way boundaries with fractal structure such as the von Koch snowflake. Anyway, the class of domains verifying the capacity density condition \eqref{3} goes beyond these common examples and contains for example sets with boundaries which support the \textit{uniform corkscrew} condition.

The regularity problem for solutions to \eqref{1} has been a long-standing problem in the PDEs theory, related to the Hilbert 19th Problem. In particular, the task to get H\"older continuity of the weak solutions under very general hypotheses on the data is a first step towards developing relevant solvability and regularity theory for \eqref{1} in the framework of various functional scales (see for instance \cite{BP2,BP3,NA} and the references therein). In case when \eqref{1} is the Euler--Lagrange equation of a given functional $\mathcal{F}$ that is the problem of regularity of the minimizers of $\mathcal{F}$ and this links \eqref{1} to important equations from differential geometry or mathematical
physics, such as Ginzburg--Landau, nonlinear Schr\"odinger, non-Newtonian fluids and so on.

The Hilbert 19th Problem has been brilliantly solved by De~Giorgi in \cite{DG} for $W^{1,2}_0$-weak solutions to \textit{linear} differential operators over Lipschitz continuous domains when $m = 2,$ $\varphi\in L^p$ with $p > n$ and $\psi\in L^q$ with $2q > n,$ and this provided the initial breakthrough in the modern theory of quasilinear equations in more than two independent variables. The De~Giorgi result was
extended to \textit{linear} equations in the \textit{non-$L^p$} settings (i.e., when a sort of \eqref{2} holds) by Morrey in \cite{M} and Lewy and Stampacchia in \cite{LS} to equations with
measures at the right-hand side, assuming 
$\nu\in \mathcal{R}^{1,\omega}$ with $\omega > n-2.$
Moving to the quasilinear equation \eqref{1}, we mention the seminal $L^p$-result of
Serrin \cite{Serrin}, which provides \textit{interior} boundedness and H\"older continuity of
the $W^{1,m}_0$-weak solutions to \eqref{1} 
in the \textit{sub-controlled} case when the nonlinearities grow as $|u|^{m-1}+|Du|^{m-1},$ and the behaviour with respect to $x$ of $\ba(x,u,Du)$ and $b(x,u,Du)$   is controlled in terms of $\varphi$ and $\psi,$ respectively, with
\begin{equation}\label{2'}
\begin{array}{lll}
\varphi\in L^{p}(\Omega) &\text{with}\ p>\frac{m}{m-1},\ & (m-1)p>n \\[4pt]
\psi\in L^{q}(\Omega) &\text{with}\ q>\frac{mn}{mn+m-n},\ &  mq>n.
\end{array}
\end{equation}

\textit{Global boundedness} of the $W^{1,m}_0$-weak solutions to \eqref{1} with general nonlinearities of \textit{controlled} growths has been obtained by Ladyzhenskaya and Ural'tseva in \cite{LU1}  under the hypotheses \eqref{2'} and for domains satisfying the \textit{measure density (A)-condition.} 
Assuming \textit{natural growths} of the data $\big($that is, $\ba(x,u,Du)=\mathcal{O}(\varphi(x)+|Du|^{m-1})$ and  $b(x,u,Du)=\mathcal{O}(\psi(x)+|Du|^m)\big)$ and \eqref{2'}, Ladyzhenskaya and Ural'tseva proved later in \cite{LU2} H\"older continuity up to the boundary for the \textit{bounded} weak solutions of \eqref{1}, and Gariepy and Ziemer  extended in \cite{GZ} their result to domains with 
\textit{$m$-thick complements}. 
It was Trudinger \cite{Tr} the first to
get \textit{global} H\"older continuity of the \textit{bounded}  solutions in the \textit{non-$L^p$} settings under the \textit{natural structure hypotheses} of Ladyzhenskaya and Ural'tseva  with
$\varphi\in L^{n/(m-1),\varepsilon},$ $\psi\in L^{n/m,\varepsilon}$ for a small $\varepsilon>0,$  while Lieberman derived in \cite{L} a
very general result on \textit{interior} H\"older continuity when $\varphi$ and $\psi$  are suitable functions and $\nu$ is a suitable measure. We refer the author also to the works by 
Rakotoson \cite{Rk}, Rakotoson and Ziemer \cite{RZ} and Zamboni \cite{Z} for various \textit{interior} regularity results regarding the problem \eqref{1}.

This paper is a natural continuation of \cite{BP1} where boundedness
has been proved for \eqref{1} with Morrey data  in the case $m=2$ under the two-sided (A) condition on  $\partial\Omega.$ Here we derive \textit{global boundedness} (Theorem~\ref{thm1}) and
\textit{H\"older continuity up to the boundary}
(Theorem~\ref{thm2}) for \textit{each} $W^{1,m}_0(\Omega)$-weak solution of the \textit{coercive}  Dirichlet problem \eqref{1} over domains with \textit{$m$-thick} complements assuming \textit{controlled growths} of the nonlinearities, Morrey type data $\varphi$ and $\psi$, and the diffusive measures $\nu$ satisfying \eqref{2} and \eqref{2-1}, respectively.
Apart from the more general class of domains considered, we extend this way the classical \textit{$L^p$-results} of Ladyzhenskaya and Ural'tseva \cite{LU1,LU2,LU} to the \textit{non-$L^p$-settings} by weakening the hypotheses on $\varphi$ and $\psi$ to the scales of \textit{Morrey type} and further considering a diffusive measure data $\nu$.   A comparison between \eqref{2} and \eqref{2'} shows
that the decrease of the degrees $p$ and
$q$ of Lebesgue integrability of the data $\varphi$ and $\psi$ is at the expense of increase of
the Morrey exponents $\lambda$ and $\mu,$ and the range of these variations is always controlled
by the relations $(m-1)p + \lambda > n$ and $mq + \mu > n.$ Indeed, in the particular case $\lambda=\mu=0$ 
and domains with exterior cone property,
our results reduce to these of Ladyzhenskaya
and Ural’tseva \cite{LU1,LU2,LU}. However, our Theorems~\ref{thm1} and \ref{thm2} generalize substantially the results in \cite{LU1,LU2,LU} because even if $(m-1)p\leq n$ and $mq\leq n,$
there exist functions $\varphi \in L^{p,\lambda}$ with
$(m-1)p+\lambda> n$  and $\psi\in L^{q,\mu}$ with $mq+\mu>n$ for which \eqref{2} hold, but  $\varphi\notin L^{p'}$ $\forall p'>n/(m-1)$ and  $\psi\notin L^{q'}$ $\forall q'>n/m$  and therefore \eqref{2'} fail. Moreover, as will be seen in Section~\ref{sec4} below,
the \textit{controlled} growths and the restrictions
\eqref{2} on the Sobolev--Morrey exponents are \textit{optimal} for the global boundedness and the subsequent H\"older continuity of the weak solutions to \eqref{1}.

The paper is organized as follows. 
In Section~\ref{sec2} we start with introducing the concept of $m$-thickness and discuss its relations to the measure density property of $\partial\Omega.$ We   list in a detailed way the hypotheses imposed on the data of \eqref{1} and state the main results of the paper. Section~\ref{sec3} collects various auxiliary results which form the analytic heart of our approach.
Of particular interest here is 
the Gehring--Giaquinta--Modica type Lemma~\ref{lemHIG}
that asserts \textit{better integrability} for the gradient of the weak solution over domains with $m$-thick complements, a particular case of which is due to Kilpel\"ainen and Koskela \cite{KK94}. The proof of the \textit{global boundedness} result 
(Theorem~\ref{thm1}) is given in Section~\ref{sec4}. Our technique relies on the De~Giorgi
 approach to the boundedness  as adapted by Ladyzhenskaya and Ural’tseva (cf. \cite[Chapter~IV]{LU}) to
 quasilinear equations. Namely, using the controlled growth assumptions, we get exact decay estimates for the total mass of the weak solution taken over its level sets. However, unlike the $L^p$-approach of Ladyzhenskaya and Ural’tseva, the
mass we have to do with is taken with respect to a positive Radon measure
$\m,$ which depends not only on the Lebesgue measure, but also on $\varphi^{\frac{m}{m-1}},$ $\psi$ and a
suitable power of the weak solution itself. Thanks to the hypotheses \eqref{2}, the measure $\m$ allows to employ very precise inequalities of trace type due to D.R.~ Adams \cite{Ad} and these lead to a bound of the $\m$-mass of $u$ in terms of the $m$-energy of $u.$ At this point we combine the controlled growth conditions 
with the better integrability of the gradient in order to estimate the $m$-energy of $u$ in terms of small multiplier of the same quantity plus a suitable power of the level set $\m$-measure. The global boundedness of the weak solution then follows by a classical result known as 
\textit{Hartman--Stampacchia maximum principle}. At the end of Section~\ref{sec4} we show sharpness of the
controlled growths hypotheses  as well as of \eqref{2} on the level of explicit examples built on quasilinear operators with $m$-Laplacean principal part. Section~\ref{sec5} is devoted to the proof of the \textit{global H\"older continuity} as claimed in
Theorem~\ref{thm2}. Indeed, the boundedness of the weak solution is guaranteed by Theorem~\ref{thm1} and the fine results obtained by Lieberman in \cite{L} apply to infer \textit{interior H\"older continuity.} To extend it up to the boundary of $\Omega,$ we adopt to our situation the approach of Gariepy and Ziemer from \cite{GZ} which relies on the Moser iteration technique in obtaining growth estimates for the gradient of the solution. The crucial step here is ensured by Lemma~\ref{lemM} which combines with the $m$-thickness condition in order to get estimate for the oscillation of $u$ over small balls centered on $\partial\Omega$ in terms of a suitable positive power of the radius. Just for the sake of simplicity, we proved Theorem~\ref{thm2} under the controlled growths hypotheses. Following the same arguments, it is easy to see that the \textit{global H\"older continuity} result still holds true for the \textit{bounded} weak solutions of \eqref{1} if one assumes the \textit{natural structure} conditions of Ladyzhenskaya and Ural'tseva instead of the controlled ones (cf. Theorem~\ref{thm3}).

\section*{Acknowledgments}

S.-S. Byun was supported by NRF-2017R1A2B2003877.
The work of D.K.~Palagachev was supported by the Italian Ministry of Education, University and Research under the Programme ``Department of Excellence'' L.~232/2016 (Grant No. CUP - D94I18000260001).
P. Shin was supported by NRF-2020R1I1A1A01066850.

\section{Hypotheses and Main Results}\label{sec2}

Throughout the paper, we will use standard notations and will assume that the functions and sets considered are measurable.

We denote by $B_\rho(x)$ (or simply $B_\rho$ if there is no ambiguity) the $n$-dimensional open ball with center $x\in\mr^n$ and radius $\rho.$ The Lebesgue measure of a measurable set $E\subset\mr^n$ will be denoted
by $|E|$ while,  for any integrable function $g$ defined on a set $E$ with $|E|>0$, its integral average is given by
$$
\overline{g}_E:=\Xint-_E g(x)\; dx = \frac{1}{|E|}\int_E g(x)\; dx.
$$

For a vector-valued function $g\colon \Omega \to \mr^n$ defined on a bounded domain $\Omega$, a real number $h \in \mr$ and an integer $i \in \{1,\cdots,n\}$,
we define the difference operator as
$$
\tau_{i,h} g(x) := g(x+he_i) - g(x), \quad \forall x \in \Omega_{|h|}
$$
where $\{e_i\}_{1\leq i \leq n}$ is the standard basis and $\Omega_{|h|}:=\{x\in \Omega : \mathrm{dist}(x,\partial\Omega)>|h|\}$.

We will denote by $C^\infty_0(\Omega)$ the space of infinitely differentiable functions over a bounded domain $\Omega\subset\mr^n$ with compact support contained in that domain, and $L^p(\Omega)$ stands for the standard Lebesgue space with a given $p\in[1,\infty].$ The Sobolev space $W^{1,p}_0(\Omega)$ is defined, as usual, by the completion of $C^\infty_0(\Omega)$ with respect to the norm
$$
\|w\|_{W^{1,p}(\Omega)}:=\|w\|_{L^p(\Omega)}+\|Dw\|_{L^p(\Omega)}
$$
for $p\in [1,\infty).$

Given $s\in[1,\infty)$ and $\theta\in[0,n],$ the Morrey space $L^{s,\theta}(\Omega)$ is the collection of all functions $w\in L^s(\Omega)$ such that 
$$
\|w\|_{L^{s,\theta}(\Omega)}:=\sup_{x_0\in\Omega,\ \rho>0}\left(\rho^{-\theta}\int_{B_\rho(x_0)\cap\Omega}|w(x)|^s\; dx\right)^{1/s} < \infty.
$$
The space $L^{s,\theta}(\Omega),$ equipped with the norm $\|\cdot\|_{L^{s,\theta}(\Omega)}$ is Banach space and 
the limit cases $\theta=0$ and $\theta=n$ give rise, respectively, to $L^s(\Omega)$ and $L^\infty(\Omega).$ 
Moreover, we denote by $\mathcal{R}^{1,\theta}(\Omega)$ to mean the collection of all signed Radon measures $\nu$ with finite mass such that
$$
\|\nu\|_{\mathcal{R}^{1,\theta}(\Omega)}:=\sup_{x_0\in\Omega,\ \rho>0}\left(\frac{|\nu|\left(B_\rho(x_0)\cap \Omega\right)}{\rho^{\theta}}\right)  < \infty.
$$

Given $\alpha\in(0,1]$ and $s\in[1,\infty),$ the fractional Sobolev space $W^{\alpha,s}(\Omega)$ is the collection of all functions $w\in L^s(\Omega)$ such that
\begin{align*}
\|w\|_{W^{\alpha,s}(\Omega)} &:= \left(\int_\Omega |w(x)|^s\; dx\right)^\frac{1}{s} + \left(\int_\Omega\int_\Omega \frac{|w(x)-w(y)|^s}{|x-y|^{n+\alpha s}}\; dxdy\right)^\frac{1}{s} \\
& = \|w\|_{L^s(\Omega)}+[w]_{\alpha,s;\Omega} <\infty,
\end{align*}
while the Nikol'skii space $\mathcal{N}^{\alpha,s}(\Omega)$ is the collection of all measurable functions $w\in L^s(\Omega)$ such that 
$$
\|w\|_{\mathcal{N}^{\alpha,s}(\Omega)} := \left(\int_\Omega |w(x)|^s\; dx\right)^\frac{1}{s} + \left(\sup_{h\neq0}\int_{\Omega_{|h|}} \frac{|\tau_{i,h}w(x)|^s}{|h|^{\alpha s}}\; dx\right)^\frac{1}{s}
$$
	is finite. We further define  $W^{0,s}(\Omega) \equiv \mathcal{N}^{0,s}(\Omega) \equiv L^s(\Omega)$ as the limit case $\alpha = 0$.
It is worth noticing that if $w\in W^{\alpha,s}(\Omega)$ with $\alpha s <n$ and $\Omega$ is a bounded Lipschitz domain, then $w \in L^\frac{ns}{n-\alpha s}(\Omega)$ and 
$$
\|w\|_{ L^\frac{ns}{n-\alpha s}(\Omega)} \leq C(n,\alpha,s,\Omega) \|w\|_{W^{\alpha,s}(\Omega)}.
$$
It is also well known that the following inclusions
$$
W^{\alpha,s}(\Omega) \subset \mathcal{N}^{\alpha,s}(\Omega) \subset W^{\alpha-\varepsilon,s}(\Omega)
$$
hold for every $\varepsilon \in (0,\alpha)$.

Let $\Omega\subset\mr^n$ be a bounded domain with $n\geq2.$ In order to set down the requirements on $\partial\Omega,$ we need to recall the concept of \textit{variational $p$-capacity} of a set for $1<p<\infty.$ Thus, given 
a compact set $K\subset\Omega,$ its $p$-capacity
is defined as
$$
\mathrm{Cap}_p(K,\Omega)=\inf_g\int_{\Omega}|Dw|^p\; dx
$$
where the infimum is taken over all functions 
$w \in C^\infty_0(\Omega)$ such that $w=1$ in $K.$ 

It is worth noticing that if $K\subset K'\subset \Omega'\subset\Omega$ then
$$
\mathrm{Cap}_p(K,\Omega)\leq \mathrm{Cap}_p(K',\Omega')
$$
and, in case of two concentric balls $B_R$ and $B_r$ with $R>\rho,$ the next formula 
$$
\mathrm{Cap}_p(\overline B_\rho,B_R)=C \rho^{n-p}
$$
is known for $p>1,$ where $C>0$ depends on $n,$ $p$ and $R/\rho$ (see \cite[Chapter~2]{HKM} for more details).

\smallskip

In the sequel we will suppose that the 
complement $\mr^n\setminus\Omega$ of $\Omega$ satisfies the next \textit{uniform $m$-thickness condition: there exist positive constants $A_\Omega$ and $r_0$ such that
\begin{equation}\label{3}
\mathrm{Cap}_m\big(\overline{B}_\rho(x)\setminus\Omega,B_{2\rho}(x)\big)\geq A_\Omega\ \mathrm{Cap}_m\big(\overline{B}_\rho(x),B_{2\rho}(x)\big)
\end{equation}
for all $x\in \mr^n\setminus\Omega$ and all $\rho\in(0,r_0).$}

\smallskip

Let us point out that replacing the capacity above with the Lebesgue measure, \eqref{3} reduces to the measure density condition (the \textit{(A)-condition} of Ladyzhenskaya and Ural'tseva)
which holds for instance when $\Omega$ supports the \textit{uniform exterior cone property.} If a given set $E$ satisfies the measure density condition then it is uniformly $p$-thick for each $p>1,$ whereas each nonempty set is uniformly $p$-thick if $p>n.$ Further on, a uniformly $q$-thick set is also uniformly $p$-thick for all $p\geq q$ and, as proved in \cite{L88}, the uniformly $p$-thick sets have a deep \textit{self-improving property} to be uniformly $q$-thick for \textit{some} $q<p,$ depending on $n,$ $p$ and the constant of the $p$-thickness. This way, it follows from \eqref{3} and 
\cite{L88} the existence of a number $P<m$ such that $\mr^n\setminus\Omega$ is uniformly $s$-thick
for each $s\in[P,m].$ Yet another example of uniformly $p$-thick sets for all $p>1$ is given by those satisfying the uniform \textit{corkscrew} condition: a set $E$ is uniformly corkscrew if there exist constants $C>0$ and $r_0>0$ such that for any $x\in E$ and any $\rho\in(0,r_0)$ there is a point $y\in B_\rho(x)\setminus E$ with the property that $B_{\rho/C}(y)\subset \mr^n\setminus E.$

\medskip

Turning back to the Dirichlet problem \eqref{1}, the nonlinearities considered are given by the Carath\'eodory maps $\ba\colon\Omega\times\mr\times\mr^n\to\mr^n$ and  $b\colon\Omega\times\mr\times\mr^n\to\mr,$ where
$\ba(x,z,\xi)=\big(a^1(x,z,\xi),\cdots,a^n(x,z,\xi)\big).$  In other words,
 the functions $a^i(x,z,\xi)$ and $b(x,z,\xi)$ are measurable with respect to $x\in\Omega$ for all $(z,\xi)\in\mr\times\mr^n$ and are continuous with respect to $z\in\mr$ and $\xi\in\mr^n$ for almost all (a.a.) $x\in\Omega.$ Moreover, we suppose:

\smallskip

$\bullet$ \textit{Controlled growth conditions:} There exist a constant $\Lambda>0$ and non-negative functions $\varphi\in L^{p,\lambda}(\Omega)$ with $p>\frac{m}{m-1},$ $\lambda\in(0,n)$ and $(m-1)p+\lambda>n,$ and $\psi\in L^{q,\mu}(\Omega)$ with $q\geq 1$, $\mu\in(0,n)$ and $mq+\mu>n,$ such that 
\begin{equation}\label{4}
\begin{cases}
|\ba(x,z,\xi)| \leq \Lambda\left(\varphi(x)+|z|^{\frac{{m^*}(m-1)}{m}}+|\xi|^{m-1}\right),\\[4pt]
|b(x,z,\xi)| \leq \Lambda\left(\psi(x)+|z|^{{m^*}-1}+|\xi|^{\frac{m({m^*}-1)}{{m^*}}}\right)
\end{cases} 
\end{equation}
for a.a. $x\in\Omega$ and all $(z,\xi)\in\mr\times\mr^n.$ Here, ${m^*}$ is the Sobolev conjugate of $m$ and is given by 
$$
{m^*} =
\begin{cases}
\frac{nm}{n-m} & \text{if}\ m<n,\\
\text{any exponent} \ {m^*}>n & \text{if}\ m=n.
\end{cases}
$$

$\bullet$ \textit{Coercivity condition:} There exists a constant $\gamma>0$ such that
\begin{equation}\label{5}
\ba(x,z,\xi)\cdot\xi\geq\gamma|\xi|^m-\Lambda|z|^{{m^*}}-\Lambda\varphi(x)^\frac{m}{m-1}
\end{equation}
for a.a. $x\in\Omega$ and all $(z,\xi)\in\mr\times\mr^n.$

\smallskip

Recall that a function $u\in W^{1,m}_0(\Omega)$ is called a \textit{weak solution} of the Dirichlet problem \eqref{1} 
if
\begin{equation}\label{6}
\int_\Omega \ba(x,u,Du)\cdot Dv\; dx + \int_\Omega b(x,u,Du)v\; dx + \int_\Omega v\; d\nu = 0
\end{equation}
for each test function $v\in W^{1,m}_0(\Omega).$ It is worth noting that
the convergence of the integrals involved in \eqref{6} for all admissible  $u$ and $v$ is ensured by \eqref{4} under the  assumptions \eqref{2} and \eqref{2-1}. In particular, the Morrey space $L^{q,\mu}(\Omega)$ and $\mathcal{R}^{1,\omega}(\Omega)$ are contained in the dual space $W^{-1,m'}(\Omega)$ of the Sobolev space $W^{1,m}_0(\Omega)$.

Throughout the paper the omnibus phrase ``\textit{known quantities}'' means that a given constant depends on the data in hypotheses \eqref{3}, \eqref{4}, \eqref{5}, \eqref{2} and \eqref{2-1} which include  $n,$ $m,$ 
${m^*},$ $\gamma,$ $\Lambda,$ $p,$ $q,$ $\lambda,$ $\mu,$ $\omega,$ $\|\varphi\|_{L^{p,\lambda}(\Omega)},$ $\|\psi\|_{L^{q,\mu}(\Omega)},$ $\|\nu\|_{\mathcal{R}^{1,\omega}(\Omega)},$  $\mathrm{diam}\,\Omega,$ $A_\Omega$ and $r_0.$ We will denote by $C$ a generic constant, depending on known quantities, which may vary within the same formula.

\medskip

Our first result claims \textit{global essential boundedness} of the weak solutions to the problem \eqref{1}.
\begin{theorem}\label{thm1}
Let $\Omega$ satisfy \eqref{3} and assume \eqref{4}, \eqref{5}, \eqref{2} and \eqref{2-1}. Then each $W^{1,m}_0(\Omega)$-weak solution to the problem \eqref{1} is globally essentially bounded. That is, there exists a constant M, depending on known quantities, on  $\|Du\|_{L^m(\Omega)}$ and on the uniform integrability of $|Du|^m,$ such that
\begin{equation}\label{10}
\|u\|_{L^\infty(\Omega)}\leq M.
\end{equation} 
\end{theorem}

An immediate consequence of Theorem~\ref{thm1} and the local properties of solutions to quasilinear elliptic equations (cf. \cite{L,Z}) is the \textit{interior H\"older continuity} of the weak solutions.
\begin{corollary}\label{crl1}
Under the hypotheses of Theorem~\ref{thm1}, each weak solution to \eqref{1} is {\em locally} H\"older continuous in $\Omega.$ That is, 
$$
\sup_{x,y\in \Omega',\ x\neq y} \frac{|u(x)-u(y)|}{|x-y|^\alpha} \leq H\quad \forall \Omega'\Subset\Omega
$$
with an exponent $\alpha\in(0,1)$ and a constant $H>0$ depending on the same quantities as $M$ in \eqref{10} and on $\mathrm{dist\,}(\Omega',\partial\Omega)$ in addition. 
\end{corollary}

What really turns out is that assumptions \eqref{3}, \eqref{4}, \eqref{5}, \eqref{2} and \eqref{2-1} are also sufficient to ensure H\"older continuity of the weak solutions \textit{up to the boundary}, and this is the essence of our second main result.
\begin{theorem}\label{thm2}
Assume \eqref{3}, \eqref{4}, \eqref{5}, \eqref{2} and \eqref{2-1}. Then each weak solution of the Dirichlet problem \eqref{1} is {\em globally} H\"older continuous in $\Omega.$ Precisely, 
$$
\sup_{x,y\in \overline\Omega,\ x\neq y} \frac{|u(x)-u(y)|}{|x-y|^\alpha} \leq H,
$$
where the exponent $\alpha\in(0,1)$ and the H\"older constant $H>0$ depend on the same quantities as $M$ in \eqref{10}. 
\end{theorem}

\section{Auxiliary Results}\label{sec3}

For the sake of completeness, we collect here some auxiliary results to be used in proving Theorems~\ref{thm1} and \ref{thm2}.

\subsection{Basic tools}

\begin{proposition}\label{MSa} 
{\em (Embeddings between Morrey spaces, see \cite{Pic})}
For arbitrary $s',s''\in [1,\infty)$ and $\theta',\theta'' \in [0,n),$ one has 
$$
L^{s',\theta'}(\Omega)\subseteq L^{s'',\theta''}(\Omega)
$$
if and only if
$$
s'\geq s''\geq 1\quad \text{and} \quad \frac{s'}{n-\theta'}\geq\frac{s''}{n-\theta''}.
$$
Moreover, if $n>\theta' \geq \theta''\geq 0$, then it holds
$$
\mathcal{R}^{1,\theta'}(\Omega) \subseteq \mathcal{R}^{1,\theta''}(\Omega) . 
$$
\end{proposition}

\begin{proposition}\label{Ati}
{\em ({Adams trace inequality}, 
see \cite{Ad})}
Let $\m$ be a positive Radon measure supported in $\Omega$ such that $\m\in \mathcal{R}^{1,\alpha_0}(\Omega)$  and let
$$
\alpha_0=\frac{s}{r}(n-r),\quad 1<r<s<\infty,\quad r<n.
$$
Then
$$
\left(\int_\Omega|v(x)|^s\; d\m\right)^{1/s}\leq C(n,s,r)\|\m\|^{1/s}_{\mathcal{R}^{1,\alpha_0}(\Omega)}\left(\int_\Omega|Dv(x)|^r\; dx\right)^{1/r}
$$ 
for all $v\in W^{1,r}_0(\Omega)$.
\end{proposition}

\begin{proposition}\label{lemEMB}{\em (see \cite[Chapter~VII]{RAAd}, \cite{KM})}
Let $\Omega' \subset\subset \Omega$ and $f\in L^s(\Omega)$ with $s\in [1,\infty).$ Suppose there exist constants $\tilde{\alpha}\in(0,1]$ and $S\geq 0$ such that
$$
\|\tau_{i,h} w\|_{L^s(\Omega')}\leq S|h|^{\tilde{\alpha}}
$$
for every $1\leq i \leq n$ and $h \in \mr$ with $0<|h|\leq d,$ where $0<d\leq \mathrm{dist}(\Omega',\partial\Omega).$
 
Then $f\in W^{\alpha,s}(\Omega')$ for every $\alpha\in (0,\tilde{\alpha})$ and
$$
\|f\|_{ W^{\alpha,s}(\Omega')} \leq C
$$
with a constant $C$ depending only on $n,$ $s,$ $d,$ $S,$ $\tilde{\alpha}-\alpha$ and $\|f\|_{L^s(\Omega')}.$
\end{proposition}

\begin{proposition}\label{GGMl}
{\em ({Gehring--Giaquinta--Modica lemma}, 
see \cite[Proposition~1.1, Chapter~V]{G})}
Let $B$ be a fixed ball and
$G\in L^s(B),$ $F\in L^{s_0}(B)$ be nonnegative functions with $s_0>s>1.$  
Suppose 
$$
\Xint-_{B_\rho}G^s(x)\; dx \leq c\left(\Xint-_{B_{2\rho}}G(x)\; dx\right)^s+\Xint-_{B_{2\rho}}F^s(x)\; dx +\theta\Xint-_{B_{2\rho}}G^s(x)\; dx
$$
for each ball $B_\rho$ of radius $\rho\in(0,\rho_0)$ such that $B_{2\rho}\subset B,$ where $0 \leq \theta < 1.$ 

Then there exist constants $C$ and $m_0\in(s,s_0],$ depending on $n,$ $c,$ $s,$ $s_0$ and $\theta,$ such that
$$
\left(\Xint-_{B_\rho}G^{m_0}(x)\; dx\right)^{1/m_0} 
\leq C\left(\left(\Xint-_{B_{2\rho}}G^s(x)\; dx\right)^{1/s}+\left(\Xint-_{B_{2\rho}}F^{s_0}(x)\; dx\right)^{1/s_0}\right).
$$
\end{proposition}

\begin{proposition}\label{HSMp}
{\em ({Hartman--Stampacchia maximum principle}, 
see \cite{HS}, \cite[Chapter~II, Lemma~5.1]{LU})}
Let $\zeta\colon\mr\to[0,\infty)$ be a  non-increasing function and suppose there exist constants $C>0,$ $k_0\geq0,$ $\delta>0$ and 
$\alpha \in [0,1+\delta]$ such that
$$
\int_k^\infty \zeta(t)\; dt \leq Ck^\alpha\zeta(k)^{1+\delta} \quad
\forall k\geq k_0.
$$
Then $\zeta$ supports the {\em finite time extinction property,} that is,
there is a number $k_{\max},$ depending on $C,$ $k_0,$ $\delta,$ $\alpha$ and $\int_{k_0}^\infty \zeta(t)\; dt,$ such that
$$
\zeta(k)=0\quad \forall k\geq k_{\max}.
$$
\end{proposition}

\begin{proposition}\label{Pj}
{\em ({John--Nirenberg lemma}, 
see \cite[Lemma~1.2]{Tr}, \cite[Theorem~7.21]{GT})}
Let $B_0$ be a ball in $\mr^n,$ $u\in W^{1,m}(B_0)$ and suppose that, for any ball $B\subset B_0$ with the same center as $B_0$ there exists a constant $K$ such that
$$\|Du\|_{L^m(B)} \leq K|B|^{\frac{n-m}{mn}}.$$
Then there exists constants $\sigma_0>0$ and $C$ depending on $K,m,n$ such that
$$\int_{B_0}e^{\sigma_0 u}\; dx \int_{B_0}e^{-\sigma_0 u}\; dx \leq C |B_0|^2. $$
\end{proposition}

\begin{proposition}\label{Pinterpolation}
{\em (see \cite[Lemma~8.23]{GT})}
Let $F$ and $G$ be nondecreasing functions in an interval $(0,R].$ Suppose that for all $\rho\leq R$ one has
$$
G(\rho/2) \leq c_0\big(G(\rho) + F(\rho)\big)
$$
for some $0<c_0<1.$ Then for any $0<\tau<1$ and $\rho\leq R$ we have
$$
G(\rho) \leq C \left(\left(\frac{\rho}{R}\right)^\alpha G(R) + F(\rho^\tau R^{1-\tau})\right)
$$
where $C=C(c_0)$ and $\alpha=\alpha(c_0,\tau)$ are positive constants.
\end{proposition}

The next result is a boundary variant of the Sobolev inequality which holds under the $m$-thickness condition.
\begin{lemma}\label{Boundary-Sobolev}
{\em (Sobolev inequality near the boundary)}
Let $\Omega$ be a bounded domain with uniformly $m$-thick complement $\mr^n\setminus\Omega$.
Consider a function $u\in W^{1,m}_0(\Omega)$ which is extended as zero outside $\Omega.$ Let $B_\rho$ be a ball of radius $\rho\in(0,r_0),$ centered at a point of $\partial\Omega$.

Then there is a constant $C=C(n,m,A_\Omega)$ such that
\begin{equation}\label{BSob}
\left( \Xint-_{B_\rho} |u(x)|^{s}\;dx\right)^{1/s}
\leq C \rho
\left( \Xint-_{B_\rho} |Du(x)|^{m}\;dx\right)^{1/m}\qquad \forall s \in[m,m^*].
\end{equation}
\end{lemma}

\begin{proof}
Combining \cite[Lemma 3.1]{KK94} and the uniform $m$-thickness condition \eqref{3}, we have the desired estimate \eqref{BSob} when $m<n$. If $m=n$, then according to \cite{L88}, there exists a constant $P<m$ such that $\mr^n\setminus \Omega$ is uniformly $P$-thick and therefore $\mr^n\setminus \Omega$ is uniformly $\tilde{s}$-thick for each $\tilde{s}\in [P,n]$.
Taking $\tilde{s}$ sufficiently close to $n$ such that $s\in [\tilde{s},\tilde{s}^*]$ for given $s\in[n,\infty)$, we have
\begin{equation*}
\left( \Xint-_{B_\rho} |u(x)|^{s}\;dx\right)^{1/s}
\leq C \rho
\left( \Xint-_{B_\rho} |Du(x)|^{\tilde{s}}\;dx\right)^{1/\tilde{s}} \leq C \rho
\left( \Xint-_{B_\rho} |Du(x)|^{n}\;dx\right)^{1/n}
\end{equation*}
in view of the H\"older inequality and this completes the proof.
\end{proof}

\subsection{Higher integrability of the gradient} 

The main result of this section, Lemma \ref{lemHIG}, provides a crucial step to obtain global boundedness of the weak solutions to \eqref{1} although it is interesting by its own. Actually, it shows that the gradient of the weak solution to controlled growths and coercive problems \eqref{1} gains better integrability over domains with $m$-thick complements.

We start with the following crucial lemma that regards higher integrability property of the solutions to the $m$-Laplace equation with diffusive measure data.
\begin{lemma}\label{w HIG}
 Let $\mathcal{B}\subset \mr^n$ be a ball containing $\Omega$ with $\mathrm{diam}\; \mathcal{B}\leq 2\; \mathrm{diam}\; \Omega,$ and  $\omega \in (n-m,n).$
Assume further that $\tilde{\nu}$ is a signed Radon measure defined on the double ball $2\mathcal{B}$ such that $\tilde{\nu} \in \mathcal{R}^{1,\omega}(2\mathcal{B}),$ and let $w \in W^{1,m}_0(2\mathcal{B})$ be a weak solution of
\begin{equation}\label{eqnw}
\mathrm{div\,}\big(|Dw|^{m-2}Dw\big) = \tilde{\nu}.
\end{equation}

Then there exist an exponent $m_1>m$ such that $w\in W^{1,m_1}(\Omega)$ and
\begin{equation*}
\|Dw\|_{L^{m_1}(\Omega)}\leq C
\end{equation*} 
with a constant $C$ depending on known quantities and $\|\tilde{\nu}\|_{\mathcal{R}^{1,\omega}(2\mathcal{B})}$ in addition.
\end{lemma}
\begin{proof} 
From \cite[Theorem 1.10]{Min1} and the fractional Sobolev embedding property, we found
$$
Dw \in L^t(\Omega,\mr^n) \quad \forall t<\frac{nm(m-1)}{(n-1)m-\omega}
$$
when $m \in [2,n]$. Therefore, noticing that 
$$
\omega > n-m \iff m<\frac{nm(m-1)}{(n-1)m-\omega},
$$ 
we obtain the higher integrability result for the case $m \in [2,n]$.

Now, let us consider the case $m \in (1,2)$. Define the vector-valued function $V\colon \mr^n \to \mr^n$ as
$$
V(\xi) := |\xi|^\frac{m-2}{2}\xi\quad \forall \xi \in \mr^n.
$$
It is well-known that for any $\xi_1, \xi_2 \in \mr^n$ one has
\begin{align}
\label{V property}
|V(\xi_1)-V(\xi_2)|^2 \leq &\ C \big(|\xi_1|^{m-2}\xi_1 -|\xi_2|^{m-2}\xi_2 \big)\cdot (\xi_1-\xi_2),\\
\label{V property 2}
C^{-1} \left(|\xi_1|^{2}+|\xi_2|^{2}\right)^\frac{m-2}{2} \leq&\ \frac{|V(\xi_1)-V(\xi_2)|^2}{\left|\xi_1-\xi_2\right|^2} \leq C \left(|\xi_1|^{2}+|\xi_2|^{2}\right)^\frac{m-2}{2},
\end{align}
where $C>0$ depends only on $m$ and $n$.

\medskip

\textit{Step 1:}
Taking $w$ as a test function in \eqref{eqnw}, we have
\begin{equation}\label{300}
\int_{2 \mathcal{B}} |Dw|^m\; dx = \int_{2 \mathcal{B}} w\; d\tilde{\nu}.
\end{equation}
The H\"older inequality gives
$$
\int_{2 \mathcal{B}} w\; d\tilde{\nu} \leq \left(\int_{2 \mathcal{B}}\; d\tilde{\nu}\right)^{1-1/s}\left(\int_{2 \mathcal{B}} |w|^s\; d\tilde{\nu}\right)^\frac{1}{s} \leq |\tilde{\nu}|(2\mathcal{B})^{1-\frac{1}{s}}  \left(\int_{2 \mathcal{B}} |w|^s\; d\tilde{\nu}\right)^\frac{1}{s}
$$
whence, applying Proposition~\ref{Ati} with
$\alpha_0 = \omega,$ $s=\frac{m\omega}{n-m}$ and $r=m,$ 
we have
\begin{align}
\label{310}
\int_{2 \mathcal{B}} w \; d\tilde{\nu} \leq&\ C \|\tilde{\nu}\|_{\mathcal{R}^{1,\omega}(2\mathcal{B})}^\frac{1}{s}|\tilde{\nu}|(2\mathcal{B}) \left(\int_{2 \mathcal{B}} |Dw|^m\; dx\right)^\frac{1}{m} \\
\nonumber
\leq&\ C \|\tilde{\nu}\|_{\mathcal{R}^{1,\omega}(2\mathcal{B})} (\mathrm{diam\; \mathcal{B}})^{\omega(1-\frac{1}{s})}\left(\int_{2 \mathcal{B}} |Dw|^m\; dx\right)^\frac{1}{m}.
\end{align}
Combining \eqref{300} and \eqref{310} and using the Young inequality, we obtain
\begin{equation}\label{estimate w}
\int_{2 \mathcal{B}} |Dw|^m\; dx \leq C \|\tilde{\nu}\|_{\mathcal{R}^{1,\omega}(2\mathcal{B})}^{\frac{m}{m-1}}.
\end{equation}

\smallskip

\textit{Step 2:}
Consider 
the concentric balls $B_{\rho}(x_0)\equiv B_{\rho}\subset B_{2\rho} \subset 2\mathcal{B}$ and let $v \in w + W^{1,m}_0(B_{2\rho})$  be the weak solution of the $m$-Laplace equation
\begin{equation}\label{eqnv}
\mathrm{div\,}\big(|Dv|^{m-2}Dv\big) = 0.
\end{equation}
Testing $w-v$ in \eqref{eqnw} and \eqref{eqnv} respectively and using \eqref{V property}, we have
\begin{align}\label{311}
& \int_{B_{2\rho}}|V(Dw)-V(Dv)|^2\; dx \\ 
\nonumber
&\qquad\leq C \int_{B_{2\rho}}\big( |Dw|^{m-2}Dw - |Dv|^{m-2}Dv \big) \cdot \left(Dw-Dv\right)\; dx \\
\nonumber
&\qquad\leq C \int_{B_{2\rho}}|w-v|\; d|\tilde{\nu}|.
\end{align}
Applying  the H\"older inequality and the Adams trace inequality (Proposition~\ref{Ati}) with
$\alpha_0=\omega,$ $s=\frac{m\omega}{n-m}$ and $r=m,$
we obtain
\begin{align}\label{312}
\int_{B_{2\rho}}|w-v|\; d|\tilde{\nu}|  \leq&\ C|\tilde{\nu}|(B_{2\rho})^\frac{s-1}{s} \left(\int_{B_{2\rho}}|w-v|^s\; d|\tilde{\nu}|\right)^\frac{1}{s} \\
\nonumber
\leq&\ C  |\tilde{\nu}|(B_{2\rho})^\frac{s-1}{s} \left(\int_{B_{2\rho}}|Dw-Dv|^m\; dx\right)^\frac{1}{m}.
\end{align}
We use now \eqref{V property 2} to get
\begin{align}\label{314}
&\int_{B_{2\rho}}|Dw-Dv|^m\; dx \\
\nonumber
&\qquad \leq C \int_{B_{2\rho}}\left(|Dw|^2+|Dv|^2\right)^\frac{m(2-m)}{4}|V(Dw)-V(Dv)|^m\; dx \\
\nonumber
&\qquad \leq C \left(\int_{B_{2\rho}} |Dw|^m\; dx\right)^\frac{2-m}{2} \left(\int_{B_{2\rho}}|V(Dw)-V(Dv)|^2\; dx\right)^\frac{m}{2}.
\end{align}
Here, the H\"older inequality and the estimate 
\begin{equation}\label{320}
\int_{B_{2\rho}}|Dv|^m\; dx \leq C \int_{B_{2\rho}}|Dw|^m\; dx.
\end{equation}
have been used in the last bound.

Therefore, combining \eqref{311}--\eqref{314}, we discover 
\begin{align}\label{comparison 1}
\int_{B_{2\rho}}|V(Dw)-V(Dv)|^2\; dx \leq&\  C |\tilde{\nu}|(B_{2\rho})^\frac{2(s-1)}{s}\left(\int_{B_{2\rho}}|Dw|^m\; dx \right)^\frac{2-m}{m} \\
\nonumber\leq&\ C \rho^\frac{2(\omega-n+m)}{m}\left(|\tilde{\nu}|(B_{2\rho})+\int_{B_{2\rho}}|Dw|^m\; dx \right),
\end{align}
since $|\tilde{\nu}|(B_{2\rho}) \leq 2\|\tilde{\nu}\|_{\mathcal{R}^{1,\omega}(2\mathcal{B})}\rho^\omega.$

\medskip

\textit{Step 3:}
For any $0<|h|<\rho$, it holds
\begin{equation}\label{comparison 2}
\begin{aligned}
\int_{B_{\rho}}|\tau_{i,h} V(Dv(x))|^2\; dx & \leq C \frac{|h|^2}{\rho^2}\int_{B_{2\rho}}|Dv|^m\; dx \leq C \frac{|h|^2}{\rho^2}\int_{B_{2\rho}}|Dw|^m\; dx
\end{aligned}
\end{equation}
as consequence of \eqref{320} and \cite[Lemma~3.2]{Min1}.

From \eqref{comparison 1} and \eqref{comparison 2} and the triangle inequality we obtain
\begin{align}\label{330}
\int_{B_\rho}|\tau_{i,h}V(Dw(x))|^2\; dx  \leq&\ C  \int_{B_{2\rho}}\big|V(Dw)-V(Dv)\big|^2\; dx \\ 
\nonumber
& \quad + C \int_{B_{\rho}}|\tau_{i,h}V(Dv(x))|^2\; dx \\
\nonumber
\leq&\ C \left(\rho^\delta + \frac{|h|^2}{\rho^2} \right) \left(|\tilde{\nu}|(B_{2\rho})+ \int_{B_{2\rho}}|Dw|^m\; dx\right)
\end{align}
with $\delta := \frac{2(\omega-n+m)}{m}>0.$

\medskip

\textit{Step 4:} Taking $|h|=\rho^\frac{2+\delta}{2}$ in \eqref{330}, we discover
$$
\int_{B_\rho}|\tau_{i,h}V(Dw(x))|^2\; dx  \leq C |h|^\frac{2\delta}{2+\delta} \left(|\tilde{\nu}|(B_{2\rho})+ \int_{B_{2\rho}}|Dw|^m\; dx\right),
$$
and the covering argument used in the proof of \cite[Lemma 6.2]{Min1} leads to
$$
\int_{\mathcal{B}}|\tau_{i,h}V(Dw(x))|^2\; dx \leq C |h|^\frac{2\delta}{2+\delta}\left(|\tilde{\nu}|(2\mathcal{B})+\int_{2\mathcal{B}}|Dw|^m\; dx\right).
$$
Therefore, we have $V(Dw) \in \mathcal{N}^{\tilde{\sigma},2}(\mathcal{B})$ with $\tilde{\sigma}=\frac{\delta}{2+\delta}>0$ and Proposition~\ref{lemEMB} gives $V(Dw) \in  W^{\sigma,2}(\mathcal{B})$ for every $0<\sigma<\tilde{\sigma}$. We then use the fractional Sobolev embedding property in order to get $V(Dw)\in L^\frac{2n}{n-2\sigma}(\mathcal{B})$ that is equivalent to $Dw \in L^{\frac{nm}{n-2\sigma}}(\mathcal{B})$. 
The claim follows by taking $m_1=\frac{nm}{n-2\sigma}.$
\end{proof}

\begin{remark}\em
In \cite{Min1}, the author established sharp fractional differentiability results for the so-called SOLAs (Solutions Obtained as Limits of Approximations), to degenerate $m$-Laplacian type equations with measure data in a quite general setting. Indeed, our argument used in the proof of Lemma~\ref{w HIG} is based on the paper \cite{Min1}. In a similar way as done in the proof of \cite[Theorem 1.10]{Min1}, that is, combining the proof of Lemma~\ref{w HIG} with bootstrap arguments, more precise differentiability results can be obtained for singular equations with diffusive measure data. We also refer to \cite{AKM, CL, Min2} and references therein about fractional differentiability results for measure data problems.
\end{remark}

\begin{lemma}\label{lemHIG}
Assume \eqref{3}, \eqref{4},  \eqref{5}, 
\eqref{2} and \eqref{2-1}, 
and let $u \in W^{1,m}_0(\Omega)$ be a weak solution of the Dirichlet problem \eqref{1}.

Then there exist exponents $m_0>m$ and $m^*_0>m^*$ such that $u\in W^{1,m_0}(\Omega)\cap L^{m^*_0}(\Omega)$ and
\begin{equation}\label{9}
\|Du\|_{L^{m_0}(\Omega)}+\|u\|_{L^{m^*_0}(\Omega)}\leq C
\end{equation} 
with a constant $C$ depending on known quantities, on $\|Du\|_{L^m(\Omega)}$ and on the uniform integrability of $|Du|^m$ in $\Omega.$
\end{lemma}
\begin{proof}
Without loss of generality, we assume that the solution $u$ and the data $\varphi$, $\psi$ and $\nu$ are extended as zero outside $\Omega.$
Let $x_0\in\Omega$ be an arbitrary point and consider 
the concentric balls $B_{\rho}\subset B_{2\rho}$ centered at $x_0$ with $2\rho\in(0,r_0).$

We first consider the case $B_{2\rho}\subset \Omega$ and take $v(x)=\eta^m(x)\left(u(x)-\overline{u}_{B_{2\rho}}\right)$  
 as test function in \eqref{6}  with $\eta\in C^\infty_0(B_{2\rho}),$ $0\leq\eta\leq1,$ $\eta\equiv1$ on $B_{\rho}$ and $|D\eta|\leq 2/\rho$. 
Having in mind \eqref{4}, \eqref{5} and the properties of $\eta$, we get the following Caccioppoli type estimate
\begin{align}\label{8'}
\Xint-_{B_{\rho}}|Du|^m+|u|^{{m^*}}\; dx &\leq C\left(\Xint-_{B_{2\rho}}\left(|Du|^m+|u|^{{m^*}}\right)^\frac{\tilde{m}}{m}\; dx\right)^\frac{m}{\tilde{m}} \\
\nonumber
&\quad + C\left(\varepsilon+\|Du\|_{L^m(B_{2\rho})}^{{{m^*}}-m}\right)\Xint-_{B_{2\rho}}|Du|^m\; dx \\
\nonumber
&\quad  + C\Xint-_{B_{2\rho}}\varphi^\frac{m}{m-1}\; dx + C \Xint-_{B_{2\rho}} \eta^m|u-\overline{u}_{B_{2\rho}}| \; d\nu_0
\end{align}
by means of an analoguous technique used in the proof of \cite[Theorem~2.2]{G}, and where  $\tilde{m} : = \max\left\{\frac{nm}{n+m},1\right\},$  $d\nu_0 := \psi(x) dx + d|\nu|.$

To estimate the last term in \eqref{8'}, we let   $w \in W^{1,m}_0(B_{2\rho})$ to be the solution of  \eqref{eqnw} with $\tilde{\nu}=\nu_0.$
Taking $v=\eta^m|u-\overline{u}_{B_{2\rho}}|$ as a test function in \eqref{eqnw}, we have
\begin{align}\label{350}
 \int_{B_{2\rho}}\eta^m|u-\overline{u}_{B_{2\rho}}|\; d\nu_0 =&\ \int_{B_{2\rho}}|Dw|^{m-2}Dw\cdot D\left(\eta^m|u-\overline{u}_{B_{2\rho}}|\right)\; dx \\
\nonumber
\leq& \varepsilon \int_{B_{2\rho}}|Du|^m\; dx + C(\varepsilon) \int_{B_{2\rho}}|Dw|^m\; dx \\
\nonumber
& \qquad +  C(\varepsilon)\int_{B_{2\rho}}|u-\overline{u}_{B_{2\rho}}|^m|D\eta|^m\; dx \\
\nonumber
\leq&\ \varepsilon \int_{B_{2\rho}}|Du|^m\; dx + C(\varepsilon) \int_{B_{2\rho}}|Dw|^m\; dx \\
\nonumber
& \qquad + C(\varepsilon)\rho^{n}\left(\Xint-_{B_{2\rho}}|Du(x)|^{\tilde{m}}\; dx\right)^{{m}/{\tilde{m}}}
\end{align}
with arbitrary $\varepsilon>0,$ and it follows from  \eqref{8'} and \eqref{350} that
\begin{align}\label{8}
\Xint-_{B_{\rho}}|Du|^m+|u|^{{m^*}}\; dx \leq&\ C\left(\Xint-_{B_{2\rho}}\left(|Du|^m+|u|^{{m^*}}\right)^\frac{\tilde{m}}{m}\; dx\right)^\frac{m}{\tilde{m}} \\
\nonumber
&\quad + C\left(\varepsilon+\|Du\|_{L^m(B_{2\rho})}^{{{m^*}}-m}\right)\Xint-_{B_{2\rho}}|Du|^m\; dx \\
\nonumber
&\quad  + C(\varepsilon)\Xint-_{B_{2\rho}}\left(\varphi^\frac{m}{m-1}+|Dw|^m\right)\; dx.
\end{align}

If $B_{2\rho}\not\subset \Omega,$ we take $v(x)=\eta^m(x)u(x)$ as test function in \eqref{6}  with $\eta\in C^\infty_0(B_{2\rho}),$ $0\leq\eta\leq1,$ $\eta\equiv1$ on $B_{\rho}$ and $|D\eta|\leq 2/\rho$.
We then use \eqref{4}, \eqref{5} and the properties of $\eta$ to get
\begin{align}\label{7n}
\int_{B_{\rho}}|Du|^m\; dx \leq& C\int_{B_{2\rho}} \varphi^\frac{m}{m-1}\; dx + C\underbrace{\int_{B_{2\rho}} |u|^{{m^*}}\; dx}_{J_1} +C\underbrace{\int_{B_{2\rho}}\eta^m|u|\; d\nu_0}_{J_2} \\
\nonumber
&\quad+C\underbrace{\int_{B_{2\rho}}\varphi|u||D\eta|\; dx}_{J_3} +C\underbrace{\int_{B_{2\rho}}|u|^{\frac{{m^*}(m-1)}{m}+1}|D\eta|\; dx}_{J_4} \\
\nonumber
&\quad+C\underbrace{\int_{B_{2\rho}}|Du|^{m-1}|u||D\eta|\; dx}_{J_5} +C\underbrace{\int_{B_{2\rho}}|Du|^{\frac{m({m^*}-1)}{{m^*}}}|u|\; dx}_{J_6}.
\end{align}
We will estimate the terms on the right-hand side of \eqref{7n} by means of the boundary Sobolev inequality \eqref{BSob}. Let us recall at this point that the $m$-thickness condition \eqref{3} and \cite{L88} ensure existence of a number $P\in(1,m)$ such that $\mr^n\setminus\Omega$ is uniformly $P$-thick. Thus,
$$
J_1  \leq  C\rho^{{m^*}\left(n\left(\frac{1}{{m^*}}-\frac{1}{m}\right)+1\right)}\left(\int_{B_{2\rho}}|Du(x)|^m\; dx\right)^{{{m^*}}/{m}-1}\int_{B_{2\rho}}|Du(x)|^m\; dx.
$$
The term $J_2$ is estimated in the same manner as \eqref{350}. Testing $v=\eta^m|u|$ in \eqref{eqnw}, we have
$$
J_2  \leq \varepsilon \int_{B_{2\rho}}|Du|^m\; dx + C(\varepsilon) \int_{B_{2\rho}}|Dw|^m\; dx + C(\varepsilon)\rho^{n}\left(\Xint-_{B_{2\rho}}|Du(x)|^{\hm}\; dx\right)^{{m}/{\hm}}.
$$
where $\hm : = \max\left\{\frac{nm}{n+m},P \right\}.$
Using the Young inequality, \eqref{BSob} and taking into account $|D\eta|\leq 2/\rho,$ we have
\begin{align*}
J_3 \leq&\ C\int_{B_{2\rho}}|u(x)|^m|D\eta(x)|^m\; dx+ C\int_{B_{2\rho}}\varphi(x)^\frac{m}{m-1}\; dx \\
\leq&\ \frac{C}{\rho^m}\int_{B_{2\rho}}|u(x)|^m\; dx+ C\int_{B_{2\rho}}\varphi(x)^\frac{m}{m-1}\; dx \\
\leq&\ C\rho^{n}\left(\Xint-_{B_{2\rho}}|Du(x)|^{\hm}\; dx\right)^{{m}/{\hm}}+ C\int_{B_{2\rho}}\varphi(x)^\frac{m}{m-1}\; dx.
\end{align*}
Similarly,
\begin{align*}
J_4 \leq&\ \frac{C}{\rho^m}\int_{B_{2\rho}}|u(x)|^m\; dx+ C\int_{B_{2\rho}}|u(x)|^{{m^*}}\; dx \\
 \leq&\ C\rho^{n}\left(\Xint-_{B_{2\rho}}|Du(x)|^{\hm}\; dx\right)^{{m}/{\hm}}+ CJ_1
\end{align*}
and
\begin{align*}
J_5 \leq&\  \varepsilon\int_{B_{2\rho}}|Du(x)|^m\; dx +
\frac{C(\varepsilon)}{\rho^m}\int_{B_{2\rho}}|u(x)|^m\; dx\\
\leq&\ \varepsilon\int_{B_{2\rho}}|Du(x)|^m\; dx +
C(\varepsilon)\rho^{n}\left(\Xint-_{B_{2\rho}}|Du(x)|^{\hm}\; dx\right)^{{m}/{\hm}}
\end{align*}
with arbitrary $\varepsilon>0.$

Finally, the H\"older inequality and \eqref{BSob} yield
\begin{align*}
J_6 \leq&\ \left(\int_{B_{2\rho}}|u(x)|^{{m^*}}\ dx\right)^{{1}/{{m^*}}} \left(\int_{B_{2\rho}}|Du(x)|^m\; dx\right)^{1-{1}/{{m^*}}} \\
\leq&\  C\rho^{n\left(\frac{1}{{m^*}}-\frac{1}{m}\right)+1}\left(\int_{B_{2\rho}}|Du(x)|^m\; dx\right)^{{1}/{m}-{1}/{{m^*}}}\int_{B_{2\rho}}|Du(x)|^m\; dx.
\end{align*}

Using the bounds for $J_1-J_6$ in \eqref{7n} leads once again to \eqref{8} with replacing $\tilde{m}$ by $\hm : = \max\left\{\frac{nm}{n+m},P\right\}$ now. However, a careful analysis of the estimate \eqref{8} above shows that these remain valid also with $\hm : = \max\left\{\frac{nm}{n+m},P\right\}$ because of $P>1$.
Therefore, \eqref{8} holds true with $\hm : = \max\left\{\frac{nm}{n+m},P\right\}$
in the both cases considered above.

Looking at \eqref{8}, we recall that $\frac{{m^*}}{m}>1.$ Thus, thanks also to the absolute continuity of the Lebesgue integral,
we can choose $\varepsilon$ and $\rho_0$ so small that if $\rho<\rho_0$ then the multiplier of $\Xint-_{B_{2\rho}}|Du(x)|^m\; dx$ at the right-hand side of \eqref{8} becomes less than ${1}/{2}.$

To apply Proposition~\ref{GGMl}, we consider the functions
$$
G(x)=\begin{cases}
 \left(|Du(x)|^m+|u(x)|^{{m^*}}\right)^\frac{\hm}{m} & \text{if}\ x\in\Omega,\\
\hfil 
0\hfil & \text{if}\ x\notin\Omega
\end{cases}
$$
and
$$
F(x)=\begin{cases} \displaystyle
\left[C(\varepsilon)\left(\varphi^\frac{m}{m-1}+|Dw|^m\right)\right]^{\frac{\hm}{m}} & \text{if}\ x\in\Omega,\\
\hfil 0 \hfil & \text{if}\ x\notin\Omega,
\end{cases}
$$
and set $s=\frac{m}{\hm},$ $s_0=\kappa\frac{m}{\hm}$ with
$\kappa\in\left(1,\min\left\{\frac{p(m-1)}{m},\frac{m_1}{m}\right\}\right)$ where $m_1$ is the constant appeared in Lemma~\ref{w HIG}. It is worth noting that the existence of such $\kappa$ is ensured by our hypotheses
$$
p>\frac{m}{m-1},\quad  m_1>m.
$$

With these settings, the inequality \eqref{8} rewrites into
$$
\Xint-_{B_\rho}G^s(x)\; dx \leq C\left(\Xint-_{B_{2\rho}}G(x)\; dx\right)^s+\Xint-_{B_{2\rho}}F^s(x)\; dx +\frac{1}{2}\Xint-_{B_{2\rho}}G^s(x)\; dx
$$
for each ball $B_\rho$ with $\rho<\rho_0$ such that $B_{2\rho}\subset B,$
where $B$ is a large enough ball containing the bounded domain $\Omega.$

At this point Proposition~\ref{GGMl} applies to ensure existence of exponents $m_0>m$ and $m^*_0>{m^*},$ and a constant $C$ such that
$$
\|Du\|_{L^{m_0}(B_\rho)}+\|u\|_{L^{m^*_0}(B_\rho)}\leq C\quad \forall \rho<\rho_0.
$$ 
The desired estimate \eqref{9}, with a constant $C$
depending on known quantities, on $\|Du\|_{L^m(\Omega)}$ and the uniform integrability of $|Du|^m,$ then 
follows by a simple covering argument.
\end{proof}

\section{Global Essential Boundedness}\label{sec4}

\subsection{Proof of Theorem~\ref{thm1}}

Let us start with the case $1<m<n.$ Consider the measure
$$
d\m:=\left(\chi(x)+\varphi(x)^\frac{m}{m-1}+\psi(x)+|u(x)|^\frac{m^2}{n-m}\right)\;dx + d|\nu|
$$
where $\chi(x)$ is the characteristic function of the domain $\Omega,$ $dx$ is the Lebesgue measure and $\varphi$ and $\psi$ are supposed to be extended as zero outside $\Omega.$

Given a ball $B_\rho$ of radius $\rho,$ we employ the  assumptions on $\varphi$, $\psi$ and $\nu$ to get
\begin{equation*}
\begin{aligned}
\int_{B_\rho}\varphi(x)^\frac{m}{m-1}\; dx &\leq \|\varphi\|_{L^{p,\lambda}(\Omega)}^\frac{m}{m-1}\rho^{n-\frac{m(n-\lambda)}{p(m-1)}}
=\|\varphi\|^\frac{m}{m-1}_{L^{p,\lambda(\Omega)}}\rho^{n-m+\left(m-\frac{m(n-\lambda)}{p(m-1)}\right)} \\
\int_{B_\rho}\psi(x)\; dx &\leq \|\psi\|_{L^{q,\mu}(\Omega)}\rho^{n-\frac{n-\mu}{q}} = \|\psi\|_{L^{q,\mu}(\Omega)} \rho^{n-m+\left(m-\frac{n-\mu}{q}\right)}
\end{aligned}
\end{equation*}
and
\begin{equation*}
|\nu|(B_\rho) \leq \|\nu\|_{\mathcal{R}^{1,\omega}(\Omega)}\rho^{\omega}
=\|\nu\|_{\mathcal{R}^{1,\omega}(\Omega)}\rho^{n-m+\left(m-n+\omega\right)}
\end{equation*}
with $m-\frac{m(n-\lambda)}{p(m-1)}>0$, $m-\frac{n-\mu}{q}>0$ and $m-n+\omega>0$ as consequence of the hypotheses $(m-1)p+\lambda>n$, $mq+\mu>n$ and $m+\omega>n.$

Further on, $u \in L^{m^*_0}(\Omega)$ by \eqref{9} 
and therefore the H\"older inequality gives
\begin{equation}\label{11}
\int_{B_\rho}|u(x)|^\frac{m^2}{n-m}\; dx \leq \|u\|^\frac{m^2}{n-m}_{L^{m^*_0}(\Omega)} \rho^{n-m+\frac{mm^*_0(n-m)-nm^2}{m^*_0(n-m)}}
\end{equation}
with $\frac{mm^*_0(n-m)-nm^2}{m^*_0(n-m)}>0$ because of $m^*_0>{m^*}=\frac{nm}{n-m}.$

This way, setting
$$
\varepsilon_0:=\mathrm{min}\left\{m-\frac{m(n-\lambda)}{p(m-1)},m-\frac{n-\mu}{q},m-n+\omega,\frac{mm^*_0(n-m)-nm^2}{m^*_0(n-m)}\right\}>0,
$$
we get
$$
\m(B_\rho)\leq K \rho^{n-m+\varepsilon_0}
$$
with a constant $K$ depending on known quantities.

For an arbitrary $k\geq1,$ we consider now the function 
$$
v(x):=\max\{u(x)-k,0\}
$$ 
and its upper zero-level set 
$$
\Omega_k:=\big\{x\in\Omega\colon\ u(x)>k\big\}.
$$ 
It is immediate that $v\equiv0$ on $\Omega\setminus\Omega_k$ and $v\in W^{1,m}_0(\Omega).$

The H\"older inequality gives
$$
\int_\Omega v(x)\; d\m = \int_{\Omega_k} v(x)\; d\m \leq \left(\int_{\Omega_k} d\m\right)^{1-{1}/{s}}\left(\int_{\Omega_k} |v(x)|^s\; d\m\right)^{{1}/{s}},
$$
whence, applying the Adams trace inequality (Proposition~\ref{Ati}) with
$$
\alpha_0=n-m+\varepsilon_0,\quad s=\frac{m(n-m+\varepsilon_0)}{n-m},\quad r=m,
$$
we get
\begin{equation}\label{12}
\int_\Omega v(x)\; d\m \leq C (\m(\Omega_k))^{1-\frac{n-m}{m(n-m+\varepsilon_0)}}\left(\int_{\Omega_k} |Dv(x)|^m\; dx\right)^{{1}/{m}}.
\end{equation}

To estimate the $L^m(\Omega_k)$-norm of the gradient $Du$ above, we will apply \eqref{4} and \eqref{5}. For, the Young inequality implies
$$
|\xi|^{\frac{nm-n+m}{n}}|z| \leq \varepsilon |\xi|^m + C(\varepsilon)|z|^{\frac{nm}{n-m}}
$$
so that the controlled growth assumptions \eqref{4} yield
$$
b(x,z,\xi)z \leq |z||b(x,z,\xi)| \leq \Lambda\Big(\varepsilon |\xi|^m + C(\varepsilon)|z|^{\frac{nm}{n-m}} + |z|\psi(x)\Big)
$$
for a.a. $x\in\Omega,$ for all $(x,\xi)\in\mathbb{R}\times\mathbb{R}^n$ and with arbitrary $\varepsilon>0$ to be chosen later. In particular, keeping in mind
$$
0<\frac{u(x)-k}{u(x)}<1\quad \textrm{a.e.}\ \Omega_k, 
$$
we have
\begin{align}\label{13}
|b(x,u(x),Du(x))v(x)|=&\ |b(x,u(x),Du(x))u(x)|\frac{u(x)-k}{u(x)} \\
\nonumber
\leq&\ \Lambda\Big(\varepsilon |Du(x)|^m + C(\varepsilon)|u(x)|^{\frac{nm}{n-m}} + |u(x)|\psi(x)\Big)
\end{align}
for a.a. $x\in\Omega_k.$

At this point, we employ $v\in W^{1,m}_0(\Omega)$ as test function in \eqref{6} and use $v\equiv0$ on $\Omega\setminus\Omega_k,$
$|Dv|=|Du|$ a.e. $\Omega_k$ and \eqref{5}, in order to conclude that
\begin{align}\label{14}
\int_{\Omega_k}|Dv(x)|^m\; dx \leq&\ C \Bigg(
\underbrace{\int_{\Omega_k}\varphi(x)^\frac{m}{m-1} \; dx}_{I_1} + \underbrace{\int_{\Omega_k} |u(x)|\psi(x) \; dx}_{I_2}\\
\nonumber &\qquad
+ \underbrace{\int_{\Omega_k} |u(x)| \; d|\nu|}_{I_3} + \underbrace{\int_{\Omega_k} |u(x)|^{\frac{nm}{n-m}}\ dx}_{I_4}\Bigg)
\end{align}
after choosing appropriately $\varepsilon$ in \eqref{13}.

It is immediate that
\begin{equation}\label{15}
I_1 \leq \m(\Omega_k).
\end{equation}

Further on, we have
\begin{equation}\label{16}
I_2 = \int_{\Omega_k}|u(x)-k+k|\psi(x)\; dx \leq \int_{\Omega_k}v(x)\psi(x)\; dx + k \int_{\Omega_k}\psi(x)\; dx,
\end{equation}
and similarily
\begin{equation}\label{16-1}
I_3 = \int_{\Omega_k}|u(x)-k+k|\; d|\nu| \leq \int_{\Omega_k}v(x)\; d|\nu| + k |\nu|(\Omega_k).
\end{equation}
Combining \eqref{16} and \eqref{16-1} we discover 
\begin{align}\label{16-2}
I_2+I_3 &\leq \int_{\Omega_k}v(x)\psi(x)\; dx + \int_{\Omega_k}v(x)\; d|\nu| + k \left(\int_{\Omega_k}\psi(x)\; dx+|\nu|(\Omega_k)\right) \\
\nonumber&\leq \int_{\Omega_k}v(x)\psi(x)\; dx + \int_{\Omega_k}v(x)\; d|\nu| + k \m(\Omega_k).
\end{align}

To estimate the first and second term on the right-hand side above,
define the measure $d\overline\m:=\psi(x)dx + d|\nu|$. 
We have 
$$\overline\m(B_\rho) \leq \m(B_\rho) \leq C\rho^{\varepsilon_0}$$ 
for each ball $B_\rho$ and therefore Proposition~\ref{Ati} can be applied with 
$$\alpha_0=n-m+\varepsilon_0,\quad s=\frac{m(n-m+\varepsilon_0)}{n-m},\quad r=m.$$
 Namely,
\begin{align*}
\int_{\Omega_k} v(x)\; d\overline\m &\leq \left(\int_{\Omega_k}\; d\overline\m\right)^{1-1/s}\left(\int_{\Omega_k} |v(x)|^s\; d\overline\m\right)^{1/s} \\
&\leq C\big(\overline\m(\Omega_k)\big)^{1-1/s}\left(\int_{\Omega_k} |Dv(x)|^m\; dx \right)^{1/m}.
\end{align*}
We use the Young inequality to estimate the last term above by
$$
\varepsilon \int_{\Omega_k} |Dv(x)|^m\; dx + C(\varepsilon)\big(\overline\m(\Omega_k)\big)^{\frac{m}{m-1}\frac{s-1}{s}}
$$
with arbitrary $\varepsilon>0.$ Moreover,
$$
\big(\overline\m(\Omega_k)\big)^{\frac{m}{m-1}\frac{s-1}{s}} \leq 
\big(\m(\Omega_k)\big)^{\frac{m}{m-1}\frac{s-1}{s}} \leq 
\m(\Omega_k)\big(\m(\Omega)\big)^{\frac{m}{m-1}\frac{s-1}{s}-1}
$$
and 
\begin{align*}
\m(\Omega)=&\ \int_\Omega \left(1 + \varphi(x)^{\frac{m}{m-1}} + \psi(x) + |u(x)|^{\frac{m^2}{n-m}}\right)\;dx + |\nu|(\Omega)\\
\leq&\ |\Omega| +C \left(
\|\varphi\|_{L^{p,\lambda}(\Omega)}^{\frac{m}{m-1}}+
\|\psi\|_{L^{q,\mu}(\Omega)}+
\|u\|_{L^{m^*_0}(\Omega)}^{\frac{m^2}{n-m}} + \|\nu\|_{\mathcal{R}^{1,\omega}(\Omega)} \right)
\end{align*}
where $C>0$ depends only on $n$, $m$, $p$, $q$, $\lambda$, $\mu$, $\omega$ and $\text{diam\,}\Omega$.
Thus, remembering \eqref{9}, $\m(\Omega)$ is bounded in terms of known quantities and $\|Du\|_{L^m(\Omega)},$ whence 
\begin{equation}\label{17}
I_2 + I_3 \leq \varepsilon  \int_{\Omega_k} |Dv(x)|^m\; dx + C(\varepsilon) k \m(\Omega_k)
\end{equation}
with arbitrary $\varepsilon>0.$

In the same manner we estimate also the last term $I_4$ of \eqref{14}. Precisely,
\begin{align*}
I_4 =&\ \int_{\Omega_k} |u(x)-k+k|^m|u(x)|^{\frac{m^2}{n-m}}\; dx \\
\leq&\ 2^{m-1}\left(\int_{\Omega_k} v^m(x)|u(x)|^{\frac{m^2}{n-m}}\; dx + k^m \int_{\Omega_k} |u(x)|^{\frac{m^2}{n-m}}\; dx \right) \\
\leq&\ 2^{m-1} \int_{\Omega_k} v^m(x)|u(x)|^{\frac{m^2}{n-m}}\; dx + 2^{m-1}k^m  \m(\Omega_k).
\end{align*}
We will estimate the first term above with the aid of the Adams trace inequality. For this goal, note that \eqref{11} implies $|u|^{\frac{m^2}{n-m}} \in L^{1,\theta}(\Omega)$ with
$$
\theta = n - m + \frac{mm^*_0(n-m)-nm^2}{m^*_0(n-m)}>n-m.
$$
Therefore, there exists an $r'<m,$ close enough to $m,$ and such that 
$$
n-m<\frac{m}{r'}(n-r')<\theta.
$$
We have then
$$
n-r'+\frac{(n-r')(m-r')}{r'} < \theta
$$
and Proposition~\ref{MSa} yields
$|u|^{\frac{m^2}{n-m}} \in L^{1,n-r'+\frac{(n-r')(m-r')}{r'}}(\Omega).$ This way, Proposition~\ref{Ati} and the H\"older inequality give 
\begin{align*}
\int_{\Omega_k} v^m(x) |u(x)|^{\frac{m^2}{n-m}}\; dx \leq&\ C\left(\int_{\Omega_k} |Dv(x)|^{r'}\; dx\right)^{m/r'}\\
 \leq&\ C|\Omega_k|^{\frac{m}{r'}-1}\left(\int_{\Omega_k}|Dv(x)|^m \; dx\right)
\end{align*}
with $C$ depending also on $\left\||u|^{\frac{m^2}{n-m}}\right\|_{L^{1,\theta}(\Omega)}$ which is bounded in terms of $\|u\|_{L^{m^*_0}(\Omega)}$ (cf. \eqref{11} and \eqref{9}). Therefore,
\begin{equation}\label{18}
I_4 \leq C\left(|\Omega_k|^{\frac{m}{r'}-1}\int_{\Omega_k}|Dv(x)|^m \; dx + k^m  \m(\Omega_k) \right)
\end{equation}
and putting \eqref{15}, \eqref{17} and \eqref{18} together, \eqref{14} takes on the form
\begin{equation}\label{19}
\int_{\Omega_k} |Dv(x)|^m \; dx \leq C\left(|\Omega_k|^{\frac{m}{r'}-1}\int_{\Omega_k}|Dv(x)|^m \; dx + k^m  \m(\Omega_k) \right)
\end{equation}
after choosing $\varepsilon>0$ small enough and remembering $k\geq1.$

We have further
$$
k^\frac{nm}{n-m}|\Omega_k| \leq \int_{\Omega_k} |u(x)|^\frac{nm}{n-m}\; dx \leq \int_\Omega |u(x)|^\frac{nm}{n-m}\; dx \leq C\|Du\|_{L^m(\Omega)}^\frac{nm}{n-m},
$$
and this means that if $k \geq k_0$ for large enough $k_0,$
depending on known quantities and on $\|Du\|_{L^m(\Omega)},$
 then the multiplier factor $C |\Omega_k|^{\frac{m}{r'}-1}$ on the right-hand side of \eqref{19} can be made less than $1/2.$ This way
\begin{equation}\label{19'}
\int_{\Omega_k} |Dv(x)|^m \; dx \leq Ck^m  \m(\Omega_k)\qquad\forall\ k \geq k_0
\end{equation}
and then \eqref{12} becomes
\begin{equation}\label{20}
 \int_{\Omega_k} v(x)\; d\m \leq Ck\big(\m(\Omega_k)\big)^{1+\frac{\varepsilon_0}{m(n-m+\varepsilon_0)}}\qquad \forall\ k \geq k_0 
\end{equation}
with $\varepsilon_0>0.$

Employing the Cavalieri principle, we have
$$
\int_{\Omega_k} v(x)\; d\m= \int_{\Omega_k}(u(x)-k)\; d\m=\int_k^\infty \m(\Omega_t)\; dt
$$
and the setting $\zeta(t):=\m(\Omega_t)$ rewrites \eqref{20} into
$$
\int_k^\infty \zeta(t)\; dt \leq Ck\zeta(k)^{1+\delta}\qquad \forall k\geq k_0,\ \delta=\frac{\varepsilon_0}{m(n-m+\varepsilon_0)}>0.
$$
It remains to apply the Hartman--Stampacchia maximum principle (Proposition~\ref{HSMp}) to conclude
$$
u(x) \leq k_{\max}\qquad \mathrm{a.e.}\ \Omega
$$
where $k_{\max}$ depends on known quantities and on  $\|Du\|_{L^m(\Omega)}$ in addition.

Repeating the above procedure with $-u(x)$ instead of $u(x),$ we get a bound from below for $u(x)$ which gives the desired estimate \eqref{10} when $m<n.$

\medskip

The claim of Theorem~\ref{thm1} in the limit case $m=n$ can be easily obtained by adapting the above procedure to the new situation. Precisely, the controlled growth condition \eqref{4} for the term $b(x,z,\xi)$ and the coercivity condition \eqref{5} have now the form
\begin{align}
\label{m=nb}
|b(x,z,\xi)| \leq&\ \Lambda\left(\psi(x)+|z|^{{m^*}-1}+|\xi|^{\frac{n({m^*}-1)}{{m^*}}}\right),\\
\label{m=na}
\ba(x,z,\xi)\cdot\xi\geq&\ \gamma|\xi|^{n}-\Lambda|z|^{{m^*}}-\Lambda\varphi(x)^\frac{n}{n-1},
\end{align}
respectively, where ${m^*}>n$ is an \textit{arbitrary} exponent,
$\varphi\in L^{p,\lambda}(\Omega)$ with $p>\frac{n}{n-1},$ $\lambda\in(0,n)$ and $(n-1)p+\lambda>n,$ and $\psi\in L^{q,\mu}(\Omega)$ with $q\geq 1$ and $\mu\in(0,n)$.

Without loss of generality, we may choose a number $m'<n,$ close enough to $n,$ and such that ${m^*}=\frac{n^2}{(n-m')(n+1)}.$ 
Setting ${(m')^*}=\frac{nm'}{n-m'},$ we have
$$
{m^*}<{(m')^*},\quad \frac{n({m^*}-1)}{{m^*}}=\frac{m'({(m')^*}-1)}{(m')^*}
$$
and therefore \eqref{m=nb} becomes
\begin{equation}\label{m=nb'}
|b(x,z,\xi)|  \leq \Lambda\left(\psi(x)+|z|^{(m')^*-1}+|\xi|^{\frac{m'((m')^*-1)}{(m')^*}}\right)
\end{equation}
for $|z|\geq 1$ and $|\xi|\geq1,$ while \eqref{m=na} takes on the form
\begin{align}\label{m=na'}
\ba(x,z,\xi)\cdot\xi\geq&\ \gamma|\xi|^{n}-\Lambda|z|^{{m^*}}-\Lambda\varphi(x)^\frac{n}{n-1}\\
\nonumber
   \geq&\ \gamma|\xi|^{m'}-\Lambda|z|^{(m')^*}-\Lambda\varphi(x)^\frac{m'}{m'-1}
\end{align}
when $|z|\geq 1$ and $|\xi|\geq1$ and where, without loss of generality, we have supposed $\varphi(x)\geq1.$

Defined now the measure
$$
d\m'=\left(\chi(x)+\varphi(x)^\frac{m'}{m'-1}+\psi(x)+|u(x)|^\frac{m'^2}{n-m'}\right)\;dx + d|\nu|,
$$
we may increase, if necessary, the value of $m',$ maintaining  it anyway less than $n,$ in order to have
$p>\frac{m'}{m'-1},$ $(m'-1)p+\lambda>n$ and $m'q+\mu>n$, and therefore
$$
\m'(B_\rho)\leq K \rho^{n-m'+\varepsilon_0}
$$
as above, with a suitable $\varepsilon_0>0.$

Considering the function $v(x)$ and the sets $\Omega_k$ as defined before, it is immediate that
$$
\int_{\{x\in\Omega_k\colon |Dv(x)|<1\}} |Dv(x)|^{m'}\;dx\leq |\Omega_k|\leq k^{m'}\m'(\Omega_k),
$$
while 
$$
\int_{\{x\in\Omega_k\colon |Dv(x)|\geq1\}} |Dv(x)|^{m'}\;dx
$$
can be estimated with the aid of \eqref{m=nb'} and \eqref{m=na'}, as already did when $1<m<n.$ That leads to the bound \eqref{19'} with $m'$ instead of $m$ and it remains to run the same procedure employed above in order to complete the proof of Theorem~\ref{thm1}.\hfill $\qed$

\subsection{Sharpness of the Hypotheses}

We will show, on the level of simple examples built on the $m$-Laplace operator, that the restrictions on the growths with respect to $u$ and $Du$ and on the Sobolev--Morrey exponents as asked in \eqref{4} and \eqref{5} are \textit{sharp} in order to have essential boundedness of the weak solutions to \eqref{1}. 

\begin{example}\label{ex1}\em (The $|u|$-growth ${m^*}-1$ of $b(x,u,Du)$ is \textit{optimal} for the boundedness.)
Let $\varkappa>{m^*}-1>m-1.$ The function 
$$
u(x):= |x|^{\frac{m}{m-\varkappa-1}} \in W^{1,m}
$$
is a local weak solution of the equation
$$
\mathrm{div\,}\left(|Du|^{m-2}Du\right)=C(n,m,\varkappa)|u|^\varkappa
$$
in the unit ball $B_1(0),$ but $u\notin L^\infty(B_1).$
\end{example}

\begin{example}\label{ex2}
\em (The gradient growth $\frac{m({m^*}-1)}{{m^*}}$ of $b(x,u,Du)$ is \textit{optimal} for the boundedness.)
Let $m<n$ and $\varkappa\in \left(\frac{m({m^*}-1)}{{m^*}},m\right).$ The function 
$$
u(x):= |x|^{\frac{m-\varkappa}{m-\varkappa-1}}-1 
$$
is a $W^{1,m}_0(B_1)$ weak solution of the Dirichlet problem for the equation
$$
\mathrm{div\,}\left(|Du|^{m-2}Du\right)=C(n,m,\varkappa)|Du|^\varkappa,
$$
but $u\notin L^\infty(B_1).$
\end{example}

\begin{example}\label{ex3}
\em
(The requirements $\varphi\in L^{p,\lambda}(\Omega)$ with  $(m-1)p+\lambda>n$ and $\psi\in L^{q,\mu}(\Omega)$ with $mq+\mu>n$ are \textit{sharp} for the boundedness.)

Let $B_R=\left\{x\in \mr^n\colon\ |x|<R<1\right\}$ and consider the functions
$$
\varphi(x):=
\frac{x}{|x|^m\big|\log|x|\big|^{m-1}}
$$
and
$$
\psi(x):=\frac{(m-n)\log|x|-m+1}{|x|^m|\log|x||^{m}}.
$$
It is immediate to check that $\varphi\in L^{\frac{n}{m-1}}\left(B_R;\mr^n\right)\subset L^{p',n-(m-1)p'}$ $\forall p'\in \left(1,\frac{n}{m-1}\right]$ but
$\varphi\notin  L^{p',n-(m-1)p'+\varepsilon}\left(B_R;\mr^n\right)$ $ \forall\varepsilon>0;$ and
$\psi\in L^{\frac{n}{m}}(B_R)\subset L^{q',n-mq'}(B_R)$ $ 
\forall q'\in \left[1,\frac{n}{m}\right],$ but
$\psi\notin  L^{q',n-mq'+\varepsilon}(B_R)$ $\forall\varepsilon>0.$
  
The \textit{unbounded} function
$$
u(x)=\log\left(\frac{\log|x|}{\log R}\right)
$$
is a $W^{1,m}_0(B_R)$-weak solution to the homogeneous Dirichlet problems of both 
$$
\mathrm{div\,}\left(|Du|^{m-2}Du-\varphi(x)\right)=0
$$
and
$$
\mathrm{div\,}\left(|Du|^{m-2}Du\right)=\psi(x).
$$
\end{example}

\section{Global H\"older Continuity}\label{sec5}

Let us start with the H\"older regularity of the weak solutions in the \textit{interior} of $\Omega$ as claimed in 
Corollary~\ref{crl1}.

\begin{proof}[Proof of Corollary~\ref{crl1}] Let $m=n.$ Then Lemma~\ref{lemHIG} implies $u\in W^{1,m_0}(\Omega)$ with $m_0>n$ and thus the interior H\"older continuity of $u$ with exponent $1-\frac{n}{m_0}$ follows from the Morrey lemma.

Suppose therefore $m<n.$ We have then ${m^*}=\frac{nm}{n-m}$  and, taking into account the essential boundedness of $u$ given by Theorem~\ref{thm1}, the structure conditions \eqref{4} and \eqref{5} can be rewritten as
\begin{align*}
|\ba(x,z,\xi)| \leq&\ \Lambda\left(\varphi'(x)+|\xi|^{m-1}\right),\\[4pt]
|b(x,z,\xi)| \leq&\ \Lambda\left(\psi'(x)+|\xi|^{\frac{m({m^*}-1)}{{m^*}}}\right) \leq \Lambda \left(\psi'(x)+|\xi|^{m}\right),\\
\ba(x,z,\xi)\cdot\xi\geq&\ \gamma|\xi|^m-\varphi''(x)
\end{align*}
for a.a. $x\in\Omega$ and all $(z,\xi)\in\mr\times\mr^n,$ where 
$$\varphi'(x)=\varphi(x)+M^{\frac{n(m-1)}{n-m}},\qquad \psi'(x)=\psi(x)+M^{\frac{nm-n+m}{n-m}}+1,
$$
$$
\varphi''(x)=\Lambda\left(M^{\frac{nm}{n-m}}+\varphi(x)^{\frac{m}{m-1}}\right).
$$

Straightforward calculations, based on the hypotheses $\varphi\in L^{p,\lambda}(\Omega),$  $p>\frac{m}{m-1},$ $(m-1)p+\lambda>n$ and $\psi\in L^{q,\mu}(\Omega),$ $q\geq1$, $mq+\mu>n,$ give
$$
\int_{B_\rho}\varphi'(x)\; dx \leq C\left(n,m,\|\varphi\|_{L^{p,\lambda}(\Omega)},M,\text{\rm diam\,}\Omega\right)
\rho^{n-\frac{n-\lambda}{p}}=C 
\rho^{n-m+1+\varepsilon_1}
$$
with $\varepsilon_1=m-1-\frac{n-\lambda}{p}>0,$
$$
\int_{B_\rho}\psi'(x)\; dx \leq C\left(n,m,\|\psi\|_{L^{q,\mu}(\Omega)},M,\text{\rm diam\,}\Omega\right)\rho^{n-\frac{n-\mu}{q}}
=C\rho^{n-m+\varepsilon_2}
$$
with $\varepsilon_2=m-\frac{n-\mu}{q}>0,$
$$
\int_{B_\rho}\varphi''(x)\; dx \leq 
C\left(n,m,\Lambda,\|\varphi\|_{L^{p,\lambda}(\Omega)},M,\text{\rm diam\,}\Omega\right)
\rho^{n-\frac{m(n-\lambda)}{p(m-1)}}=C
\rho^{n-m+\varepsilon_3}
$$
with $\varepsilon_3=m-\frac{m(n-\lambda)}{p(m-1)}>0,$ and
\begin{equation*}
|\nu|(B_\rho) \leq \|\nu\|_{\mathcal{R}^{1,\omega}(\Omega)}\rho^{n-m+\left(m-n+\omega\right)} = C \rho^{n-m+\varepsilon_4}
\end{equation*}
with $\varepsilon_4=m-n+\omega>0$.

At this point, the claim of Corollary~\ref{crl1} follows from the Harnack inequality proved by Lieberman (see \cite[Theorem~4.1]{L} and \cite[Theorem~2.2]{Tr}) and standard covering arguments.
\end{proof}

To proceed further with the more delicate question of H\"older continuity \textit{up to the boundary} of $\Omega,$ we need the following result ensuring suitable growth estimate for the gradient over small balls.
\begin{lemma}\label{lemM}
Assume \eqref{4}, \eqref{5}, \eqref{2} and \eqref{2-1}, and
let $u$ be a weak solution to the problem \eqref{1} extended az zero outside $\Omega.$ 
Let $B_\rho$ be a ball of radius 
$\rho\in(0,\mathrm{diam}\,\Omega)$ and centered at a point of $\partial\Omega,$ and $\eta \in C^\infty_0(B_{\rho/2})$ with $|D\eta|\leq c/\rho.$ 
Define $\tilde{u}(x):=\max\{u(x),0\}$, $M(\rho):=\esssup_{x \in B_\rho} \tilde{u}(x)$, $A(\rho) : = \rho  + \|\nu\|^\frac{1}{m}_{\mathcal{R}^{1,\omega}(B_\rho)} + {\|\varphi\|}_{L^{p,\lambda}(B_\rho)}^\frac{1}{m-1} + {\|\psi\|}_{L^{q,\mu}(B_\rho)}^\frac{1}{m}$ and $w(x)^{-1}:=M(\rho)+A(\rho)-\tilde{u}(x).$

Then, there exists a constant $C$ depending on the same quantities as $M$ in \eqref{10}, such that
$$
\int_{B_{\rho/2}} {|D(\eta w^{-1})|}^m \;dx\leq  C\big(M(\rho)+A(\rho)\big){\left(M(\rho)-M\left(\frac{\rho}{2}\right)+A(\rho)\right)}^{m-1}\rho^{n-m}.
$$
\end{lemma}
\begin{proof}
Given a ball $B_\rho$ of radius $\rho$ centered at a boundary point, let us consider the measure
$$
d\m:=\chi(x)\left(\varphi(x)^\frac{m}{m-1}+\psi(x)\right)dx + \chi(x) d|\nu|
$$
where, as before, $\chi(x)$ is the characteristic function of the domain $\Omega\cap B_\rho,$ $dx$ is the Lebesgue measure and $\varphi$, $\psi$ and $\nu$ are supposed to be extended as zero outside $\Omega.$

For an arbitrary ball $B_r(y)$ of radius $r$ and center $y\in \mr^n$, we have
\begin{equation*}
\begin{aligned}
\int_{B_r(y)}\chi(x)\varphi(x)^\frac{m}{m-1}\; dx &\leq \|\varphi\|_{L^{p,\lambda}(B_\rho)}^\frac{m}{m-1}r^{n-\frac{ m(n-\lambda)}{p(m-1)}}
\leq A(\rho)^{ m}r^{n-\frac{m(n-\lambda)}{p(m-1)}}, \\
\int_{B_r(y)}\chi(x)\psi(x)\; dx & \leq \|\psi\|_{L^{q,\mu}(B_\rho)} r^{n-\frac{n-\mu}{q}} \leq A(\rho)^m r^{n-\frac{n-\mu}{q}}
\end{aligned}
\end{equation*}
and
\begin{equation*}
\int_{B_r(y)}\chi(x) d|\nu| \leq \|\nu\|_{\mathcal{R}^{1,\omega}(B_\rho)}r^{\omega} \leq A(\rho)^{m}r^{\omega}.
\end{equation*}

Therefore, setting
\begin{equation}\label{alpha}
\alpha := \min\left\{n-\frac{m(n-\lambda)}{p(m-1)}, n-\frac{n-\mu}{q}, \omega \right\}>n-m, 
\end{equation}
we get 
\begin{equation}\label{M_est1}
\m(B_r(y)) \leq C A(\rho)^{m}r^{\alpha}
\end{equation}
with a constant $C$ depending on known quantities. In particular, we have
\begin{equation}\label{M_est2} 
\m(B_\rho) \leq C A(\rho)^{m}\rho^{\alpha}.
\end{equation}

Now, we use $v := \eta^m e^{\frac{\Lambda}{\gamma}\tilde{u}}v_0$ as test function 
in \eqref{6},  where $\eta \in C^\infty_0(B_{\rho}),$  $v_0 = w^\beta - (M(\rho)+A(\rho))^{-\beta}$ and $\beta>0$ is a parameter under control. Having in mind that $|D\tilde{u}|=|Du|$ a.e. $\{x\in\Omega\colon u(x)>0\},$ we obtain
\begin{align*}
0= & \int_{E} \eta^me^{\frac{\Lambda}{\gamma}\tilde{u}}\left(\beta w^{\beta+1}+\frac{\Lambda}{\gamma}v_0\right)\ba(x,u(x),Du(x))\cdot Du\;dx\\
&\quad
 + \int_{E} m\eta^{m-1}e^{\frac{\Lambda}{\gamma}\tilde{u}}v_0\ \ba(x,u(x),Du(x)) \cdot D\eta\;dx 
\\ 
&\quad  + \int_{E} \eta^m e^{\frac{\Lambda}{\gamma}\tilde{u}} b(x,u(x),Du(x))v_0\;dx + \int_{E} \eta^m e^{\frac{\Lambda}{\gamma}\tilde{u}}v_0\;d\nu 
\end{align*}
where $E:=B_\rho\cap\{x\in\Omega\colon u(x)>0\}.$
Using that
$$
|Du|^\frac{m({m^*}-1)}{{m^*}} \leq  |Du|^m + 1
$$
in view of the Young inequality, we get 
\begin{align*}
\gamma\beta\int_E\! \eta^m e^{\frac{\Lambda}{\gamma}\tilde{u}} w^{\beta+1}|Du|^m\; dx \leq&\ \Lambda m\int_E\! \eta^{m-1}e^{\frac{\Lambda}{\gamma}\tilde{u}}v_0\! \left(|Du|^{m-1}+|u|^\frac{{m^*}(m-1)}{m}+\varphi\right)\!|D\eta|\; dx\\
 &\ +\Lambda\int_E \eta^m  e^{\frac{\Lambda}{\gamma}\tilde{u}} v_0 \left(|u|^{{m^*}-1} + \psi + 1\right)\; dx\\
&\ +\Lambda\int_E \eta^m e^{\frac{\Lambda}{\gamma}\tilde{u}}\left(\beta w^{\beta+1} + \frac{\Lambda}{\gamma}v_0 \right)\left(|u|^{m^*}+\varphi^\frac{m}{m-1}\right)\; dx \\
 &\ +\int_E \eta^m  e^{\frac{\Lambda}{\gamma}\tilde{u}} v_0 \; d|\nu|
\end{align*}
as consequence of \eqref{4} and \eqref{5}.
Since $v_0\leq w^\beta$, $|\tilde{u}| \leq |u| \leq M$ and 
$$
w^{-1}\leq C\left(\|u\|_{L^\infty(\Omega)}+{\|\varphi\|}_{L^{p,\lambda}(\Omega)}^\frac{1}{m-1} + {\|\psi\|}_{L^{q,\mu}(\Omega)}^\frac{1}{m}+{\|\nu\|}_{\mathcal{R}^{1,\omega}(\Omega)}^\frac{1}{m}+\mathrm{diam}\, \Omega\right),
$$ 
it follows
\begin{align}\label{30}
\beta\int_E \eta^m w^{\beta+1}|Du|^m\; dx \leq&\ C\int_E \eta^{m-1} w^\beta |Du|^{m-1}|D\eta|\; dx\\
\nonumber
&\ +C\int_E \eta^{m-1} w^\beta \left(\varphi+1\right)|D\eta|\; dx\\
\nonumber
&\ +C\int_E \eta^m w^{\beta+1} \left(\psi + 1\right)\; dx+C\int \eta^m w^{\beta+1} \; d|\nu|\\ 
\nonumber
&\ + C(1+\beta)\int_E \eta^m w^{\beta+1}\left(\varphi^\frac{m}{m-1} + 1\right)\; dx
\end{align}
where the constant $C$ depends on known quantities and on $\|Du\|_{L^m(\Omega)}$ through $M$ in \eqref{10}.
We then apply the Young inequality to the first and second term in the right-hand side of \eqref{30} to get
\begin{align*}
\int_E \eta^{m-1} w^\beta|Du|^{m-1}|D\eta|\; dx \leq&\ \varepsilon \int_E \eta^m w^{\beta+1}|Du|^m\; dx\\
 &\ + C\varepsilon^{1-m}\int_E w^{\beta-m+1}|D\eta|^m\; dx,\\[4pt]
\int_E \eta^{m-1} w^\beta \left(\varphi+1\right)|D\eta|\; dx \leq&\ \beta \int_E \eta^mw^{\beta+1}\left(\varphi^\frac{m}{m-1}+1\right)\; dx\\
&\ + C\beta^{1-m}\int_E w^{\beta-m+1}|D\eta|^m\; dx
\end{align*}
for any $\varepsilon>0.$ Choosing $\varepsilon=\frac{\beta}{2C}$ with appropriate $C$ above, we obtain from \eqref{30}
\begin{align}\label{31}
\beta\int_E \eta^m w^{\beta+1}|Du|^m\; dx  \leq&\ C(1+\beta)\left(\int_E \eta^m w^{\beta+1}\; d\m+ \int_E \eta^m w^{\beta+1}\; dx \right) \\
\nonumber
&\qquad + C\beta^{1-m}\int_E w^{\beta-m+1}|D\eta|^m\; dx.
\end{align}

Take now $\beta=m-1$ in \eqref{31}. We have $M(\rho)-\overline{u}\geq0$ whence $w\leq A(\rho)^{-1}$
and the Poincar\'e inequality yields
$$
\int \eta^m w^m\; dx \leq A(\rho)^{-m} \int \eta^m\; dx \leq C \int |D\eta|^m\; dx
$$
and
$$
\int_E \eta^m w^m d\m \leq {A(\rho)}^{-m} \int_E \eta^m d\m.
$$
To estimate the term on the right-hand side above, we will distinguish between the cases $1<m<n$ and $m=n.$ Thus, if $1<m<n$, we apply the Adams trace
inequality from Proposition~\ref{Ati} with
$\alpha_0=\alpha$, $s=\frac{\alpha m}{n-m}$ and $r=m$, and where $\alpha$ is taken from \eqref{alpha}. Keeping in mind \eqref{M_est1} and \eqref{M_est2}, we have
\begin{align*}
{A(\rho)}^{-m} \int_E \eta^m d\m \leq&\ 
 {A(\rho)}^{-m} \left(\int_E \eta^s d\m\right)^{{m}/{s}}\left(\int_E d\m\right)^{1-{m}/{s}} \\
\leq&\ {A(\rho)}^{-m}\left(C{A(\rho)}^\frac{m^2}{s}\int_E |D\eta|^m\; dx\right) \big(C{A(\rho)}^m\rho^{\alpha}\big)^{1-{m}/{s}}\\
=&\ C\rho^{\alpha-n+m}\int_E |D\eta|^m\; dx
\leq C\int_E |D\eta|^m\; dx,
\end{align*}
where $\alpha-n+m>0$ and $0<\rho<\mathrm{diam}\,\Omega$ have been used in the last bound.

If instead $m=n,$ we employ once again
Proposition~\ref{Ati}, but 
$\alpha_0=\alpha,$ $s=n$ and $r=\frac{n^2}{\alpha+n}$ now. Thus
\begin{align*}
{A(\rho)}^{-n} \int_E \eta^n d\m 
\leq&\ C \left(\int_E |D\eta|^r\; dx\right)^{{n}/{r}} 
\leq C \left(\int_E |D\eta|^n\; dx\right) \left(\int_E \; dx \right)^{{n}/{r}-1}\\
\leq&\ C\rho^{\alpha}\int_E |D\eta|^n\; dx \leq C\int_E |D\eta|^n\; dx
\end{align*}
thanks to $\alpha>0$, $0<\rho<\mathrm{diam}\,\Omega$ and $|E|\leq C \rho^n$, and \eqref{M_est1}.  

This way, 
$$
\int_E \eta^m w^m \; d\m  \leq C\int_E |D\eta|^m\; dx
$$
and \eqref{31} with $\beta=m-1$ becomes
$$
\int_E \eta^m|D(\log w)|^m\; dx \leq C\int_E |D\eta|^m\; dx
$$
for each $0 \leq \eta \in C^\infty_0(B_{\rho}).$ 

Choosing appropriately $\eta,$ we are in a position to apply Proposition~\ref{Pj} that asserts existence of constants $C$ and $\sigma_0>$ such that
\begin{equation}\label{32}
\int_{B_{3\rho/4}} w^{-\sigma}\;dx \int_{B_{3\rho/4}} w^\sigma\;dx \leq C\rho^{2n}\qquad \text{for all}\ |\sigma|\leq \sigma_0.
\end{equation}

Consider now the cases $\beta\neq m-1$ in \eqref{31}. For, we multiply the both sides of \eqref{31} by $\beta^{m-1}$ which rewrites it as
\begin{align}\label{31'}
\beta^m \int_E \eta^m w^{\beta+1} |Du|^m\;dx \leq&\  C(1+\beta^m) \left(\int_E \eta^m w^{\beta+1}\; d\m+ \int_E \eta^m w^{\beta+1}\; dx \right) \\
\nonumber
&\qquad + C \int_E w^{\beta-m+1} |D\eta|^m\;dx.
\end{align}
Setting $\beta=mt + m -1>0$, we have 
$$
|D(\eta w^t)|^m \leq 2^{m-1}\left( w^{mt}|D\eta|^m+|t|^m\eta^m w^{mt+m} |Du|^m\right)
$$
and the use of \eqref{31'} with $\beta=mt + m -1$
gives
\begin{align*}
\int_E |D(\eta w^t)|^m\; dx &\leq\  C\left(1+
\dfrac{|t|^m}{(mt+m-1)^m}\right)\int_E
w^{mt}|D\eta|^m\; dx\\
&\qquad + C\left(|t|^m+
\dfrac{|t|^m}{(mt+m-1)^m}\right)
\int_E \eta^m w^{mt+m}\; d\m\\
&\qquad + C\left(|t|^m+
\dfrac{|t|^m}{(mt+m-1)^m}\right)
\int_E \eta^m w^{mt+m}\; dx.
\end{align*}

We have
$$
\dfrac{t}{mt+m-1}<\dfrac{1}{m}\quad \forall t>0,
$$
while $|t|^m<\left(\frac{m-1}{m}\right)^m$ and $\frac{|t|^m}{(mt+m-1)^m}$ is a positive and decreasing function
whenever $t\in\left(\frac{1-m}{m},0\right).$ 

Thus, defining
$$
N(t):=
\begin{cases}
1+t^m & \text{if}\ t>0,\\
1+\dfrac{|t|^m}{(mt+m-1)^m} & \text{if}\ \dfrac{1-m}{m}<t\leq0,
\end{cases}
$$
the last bound takes on the form
\begin{equation}\label{33}
\begin{aligned}
\int_E |D(\eta w^t)|^m\; dx  \leq   C N(t)\bigg(\int_E (\eta w^{t+1})^m\; d\m  + \int_E w^{mt}(\eta w+|D\eta|)^m\; dx \bigg).
\end{aligned}
\end{equation}

 In order to estimate the first term on the right-hand side of \eqref{33} we will employ once again Proposition~\ref{Ati}.
Since $\frac{mn}{m+\alpha}<m$, we can take $r_0>1$ such that $\max\left\{1,\frac{mn}{m+\alpha}\right\}<r_0<m$. Taking
$\alpha_0=\alpha,$ $s=\frac{\alpha r_0}{n-r_0}$ and $r=r_0$ in Proposition~\ref{Ati}, remembering $w\leq A(\rho)^{-1}$, \eqref{M_est2} and noticing that
$$
s=\frac{\alpha r}{n-r} >m \iff \frac{mn}{m+\alpha}<r,
$$
which implies $s>m>r,$ we discover
\begin{align*}
 \int_E \eta^m w^{mt+m}\;d\m \leq&\  A(\rho)^{-m} \int_E \left(\eta w^t\right)^m d\m\\ 
\leq&\  A(\rho)^{-m}\m(B_\rho)^{1-{r}/{s}}\left(\int_E \left[\left(\eta w^t\right)^\frac{m}{r}\right]^s d\m\right)^{{r}/{s}}\\ 
\leq&\ C \rho^{\alpha(1-r/s)} \int_E \left|D(\eta w^t)^\frac{m}{r}\right|^r\; dx \\ 
\leq&\ C \int_E (\eta w^t)^{m-r}\left|D(\eta w^t)\right|^r\; dx
\end{align*}
where $0<\rho<\mathrm{diam}\,\Omega$ has been used in the last bound. Applying the Young inequality, we obtain
\begin{equation*}
 \int_E \eta^m w^{mt+m}\;d\m \leq \varepsilon\int_E \left|D(\eta w^t)\right|^m\; dx + C\varepsilon^{\frac{r}{r-m}} \int_E (\eta w^t)^{m}\; dx.
\end{equation*}

Further on, choosing $\varepsilon=\frac{1}{2CN(t)}$ above and having in mind $N(t)\geq1,$ we get from 
 \eqref{33} 
\begin{equation}\label{35}
\int_E |D(\eta w^t)|^m\; dx \leq CK(t)\int_E w^{mt}\left(\eta w + {|D\eta|}\right)^m\; dx
\end{equation}
with
$$
K(t):=\big(N(t)\big)^\theta \quad \mathrm{for}\quad \theta = \frac{m}{m-r_0}.
$$

Let us take now  a cut-off function $\eta \in C^\infty_0(B_r)$ such that $0\leq\eta\leq1,$ $\eta=1$ on $B_s$ and $|D\eta|\leq\frac{c}{r-s}$ where $0<s<r\leq\rho.$ Employing the Sobolev inequality, we get 
\begin{align}\label{36}
\left(\int_{B_s} \left(w^t\right)^{m^*} \; dx\right)^{{m}/{{m^*}}} =&\
\left(\int_{B_s} \left(\eta w^t\right)^{m^*} \; dx\right)^{{m}/{{m^*}}}\\
\nonumber
 \leq&\ C \left( \int_{B_r} |D(\eta w^t)|^\frac{{m^*} n}{{m^*}+n}\; dx \right)^{{m({m^*}+n)}/{({m^*} n)}}\\
 \nonumber
\leq&\ C\left(\int_{B_r} |D(\eta w^t)|^m \; dx \right)|B_r|^{{m({m^*}+n)}/{({m^*} n)}-1} \\ 
\nonumber
\leq&\ \frac{CK(t)}{(r-s)^m}\rho^\frac{m({m^*}+n)-{m^*} n}{{m^*}} \int_{B_r} w^{mt}\; dx
\end{align}
from \eqref{35}  since $\eta w \leq \frac{1}{\rho} \leq \frac{1}{r-s}$.

Let $\rho_k=\rho\left(\frac{1}{2}+\frac{1}{2^{k+2}}\right)$ for $k=0,1,\ldots$ and
let $t_0>0$ be any number such that $mt_0\leq\sigma_0$ with $\sigma_0$ appearing in \eqref{32}. Making use of the simple inequalities $e^{2mt}\geq (1+t)^{2m}\geq 1+ t^{2m}$ valid for all $t\geq0$ and all $m\geq \frac{1}{2},$ and remembering the properties of the function $K(t),$ we have
$$
K(t) = \left(1+t^m\right)^\theta \leq e^{2m\theta\sqrt{t}}\quad \forall t>0.
$$
Thus, taking $t=t_0\left(\frac{{m^*}}{m}\right)^k>0$ and using \eqref{36} with
$s=\rho_{k+1}$ and $r=\rho_k,$ we get
\begin{align*}
& \left(\int_{B_{\rho_{k+1}}} w^{mt_0\left(\frac{{m^*}}{m}\right)^{k+1}}\; dx\right)^{\left({m}/{{m^*}}\right)^{k+1}}\\
&\qquad \leq C^{mk\left(\frac{m}{{m^*}}\right)^k} \rho^{\frac{n(m-{m^*})}{{m^*}}\left(\frac{m}{{m^*}}\right)^k} e^{2m\theta\sqrt{t_0}\left(\frac{m}{{m^*}}\right)^\frac{k}{2}}\left(\int_{B_{\rho_k}} w^{mt_0\left(\frac{{m^*}}{m}\right)^k}\; dx\right)^{\left({m}/{{m^*}}\right)^k}
\end{align*}
for $k=0,1,\ldots.$ Iteration of these inequalities from $0$ to $N\in\mathbb{N}$ yields
\begin{align*}
& \left(\int_{B_{\rho_{N+1}}} w^{mt_0\left(\frac{{m^*}}{m}\right)^{N+1}}\; dx\right)^{\left({m}/{{m^*}}\right)^{N+1}}\\
&\qquad \leq 
C^{m\sum_{k=0}^N k\left(\frac{m}{{m^*}}\right)^{k}}
e^{2m\theta\sqrt{t_0}\sum_{k=0}^N \left(\frac{m}{{m^*}}\right)^{\frac{k}{2}}}
\rho^{\frac{n(m-{m^*})}{{m^*}}\sum_{k=0}^N \left(\frac{m}{{m^*}}\right)^{k}}
\int_{B_{\rho_0}} w^{mt_0}\; dx,
\end{align*}
and passage to the limit as $N\to+\infty$
gives
\begin{equation}\label{37}
\esssup_{B_{{\rho}/{2}}} w^{mt_0}=
\left(M(\rho)-M\left(\frac{\rho}{2}\right)+A(\rho)\right)^{-mt_0} \leq C\rho^{-n}\int_{B_{3\rho/4}} w^{mt_0}\; dx.
\end{equation}
This way, it follows from \eqref{32} and \eqref{37} that
\begin{equation}\label{38}
\rho^{-n}\int_{B_{3\rho/4}} w^{-mt_0}\; dx \leq C\left(M(\rho)-M\left(\frac{\rho}{2}\right)+A(\rho)\right)^{mt_0}
\end{equation}
for any $t_0>0$ such that $0<mt_0\leq\sigma_0.$

To proceed further, we set $\sigma=-mt_1$ where
$t_1=t_1(m,n)<0$ will be chosen in the sequel.
Let $0<\sigma<m-1$ and define $\sigma_1=\sigma\left(\frac{m}{{m^*}}\right)^\kappa$ where $\kappa$ is a positive integer for which $m-1 \leq \sigma_0\left(\frac{{m^*}}{m}\right)^\kappa$ and $\sigma_0$ is taken from \eqref{32}.   Obviously, $0<\sigma_1\left(\frac{{m^*}}{m}\right)^k\leq\sigma$ for $0\leq k\leq\kappa.$ Thus 
$\frac{1-m}{m}<-\frac{\sigma}{m}\leq-\frac{\sigma_1}{m}\left(\frac{{m^*}}{m}\right)^k$ and therefore
$$
K\left(-\frac{\sigma_1}{m}\left(\frac{{m^*}}{m}\right)^k\right)\leq K\left(-\frac{\sigma}{m}\right)
$$ 
for $0\leq k\leq\kappa$ since $K$ is a decreasing function on $\left(\frac{1-m}{m},0\right).$

We take now $\rho_k=\frac{\rho}{4}\left(3-\frac{k}{\kappa+1}\right)$ for $0\leq k \leq \kappa+1$ and  apply \eqref{36} with $s=\rho_{k+1},$ $r=\rho_k$ and $t=-\frac{\sigma_1}{m}\left(\frac{{m^*}}{m}\right)^k$ in order to get
\begin{align*}
&\left(\int_{B_{\rho_{k+1}}} w^{-\sigma_1\left(\frac{{m^*}}{m}\right)^{k+1}}\; dx\right)^{\left({m}/{{m^*}}\right)^{k+1}}\\
&\qquad\leq \left(CK\left(-\frac{\sigma}{m}\right)4^m(\kappa+1)^m \rho^{\frac{n(m-{m^*})}{{m^*}}} \right)^{\left(\frac{m}{{m^*}}\right)^k}\left(\int_{B_{\rho_k}} w^{-\sigma_1\left(\frac{{m^*}}{m}\right)^k}\; dx\right)^{\left({m}/{{m^*}}\right)^k}
\end{align*}
for $0\leq k\leq\kappa.$ Iteration of these inequalities for $0\leq k\leq\kappa$ gives
\begin{align*}
&\left(\int_{B_{\rho/2}} w^{-\sigma_1\left(\frac{{m^*}}{m}\right)^{\kappa+1}}\; dx\right)^{\left({m}/{{m^*}}\right)^{\kappa+1}}\\
&\qquad \leq \left(CK\left(-\frac{\sigma}{m}\right)4^m(\kappa+1)^m 
\right)^{\sum_{k=0}^\kappa\left(\frac{m}{{m^*}}\right)^k}\rho^{\frac{n(m-{m^*})}{{m^*}}\sum_{k=0}^\kappa\left(\frac{m}{{m^*}}\right)^k}
\int_{B_{3\rho/4}} w^{-\sigma_1}\; dx,
\end{align*}
whence
$$
\rho^{-n}\int_{B_{\rho/2}} w^{-\sigma_1\left(\frac{{m^*}}{m}\right)^{\kappa+1}}\; dx \leq C\left(\rho^{-n}\int_{B_{3\rho/4}} w^{-\sigma_1}\; dx\right)^{\left({{m^*}}/{m}\right)^{\kappa+1}}.
$$

Remembering $0<\sigma_1<\sigma_0,$ we obtain from \eqref{38} with $\sigma_1=mt_0$ that
\begin{equation}\label{39}
\rho^{-n}\int_{B_{\rho/2}} w^{t_1{m^*}}\; dx \leq C\left(M(\rho)-M\left(\frac{\rho}{2}\right)+A(\rho)\right)^{-t_1{m^*}}
\end{equation}
for each $t_1<0$ such that $0<-mt_1<m-1.$

Take now $v = \eta^m e^{\frac{\Lambda}{\gamma}\tilde{u}}\tilde{u}$ as a test function in \eqref{6}. Keeping in mind \eqref{4}, \eqref{5} and \eqref{10}, we obtain
\begin{align}\label{40}
& \int_{B_{\rho/2}} \eta^m{|Du|}^m\; dx \leq CM(\rho)\underbrace{\int_{B_{\rho/2}} \eta^{m-1}{|Du|}^{m-1}|D\eta|\; dx}_{S_1}\\ 
\nonumber
&\quad+ C{M(\rho)}\underbrace{ \int_{B_{\rho/2}} \eta^{m-1}|D\eta|(1+\varphi)\; dx}_{S_2} + C \underbrace{\int_{B_{\rho/2}} \eta^m \; d\m}_{S_3}+ C \underbrace{\int_{B_{\rho/2}} \eta^m \; dx}_{S_4}.
\end{align}

Fix $\eta\in C^\infty_0\left(B_{\rho/2}\right)$ such that $0\leq\eta\leq1$ and $|D\eta|\leq c/\rho.$ To estimate $S_1$, we take a $t_2<0$ such that $1<(1+t_2)m<\frac{{m^*}}{m}.$ Thus
\begin{align*}
S_1 &= \int_{B_{\rho/2}}\big(\eta w^{1+t_2}|Du|\big)^{m-1}\big(w^{-(1+t_2)(m-1)}|D\eta|\big)\; dx\\
& = C\int_{B_{\rho/2}}\big(\eta |D(w^{t_2})|\big)^{m-1}\big(w^{-(1+t_2)(m-1)}|D\eta|\big)\; dx\\
&\leq C\left(\int_{B_{\rho/2}}\big(\eta |D(w^{t_2})|\big)^m\; dx\right)^{1-{1}/{m}}
\left(\int_{B_{\rho/2}} \big(w^{-(1+t_2)(m-1)}|D\eta|\big)^m\; dx\right)^{{1}/{m}}.
\end{align*}
The two terms above will be estimated with the aid of \eqref{35} and \eqref{39}, respectively. Precisely,
applying \eqref{35} with $t=t_2$ and \eqref{39} with $t_1=\frac{mt_2}{m^*}$, we have
\begin{align*}
\int_{B_{\rho/2}}\eta^m |D(w^{t_2})|^m\; dx & \leq C\int_{B_{\rho/2}}|D(\eta w^{t_2})|^m\; dx+\int_{B_{\rho/2}}\big(w^{t_2}|D\eta|\big)^m\; dx\\
& \leq C\int_{B_{\rho/2}}w^{mt_2}\big(\eta^m+|D\eta|^m\big)\; dx\\
&\leq C\left(M(\rho)-M\left(\frac{\rho}{2}\right)+A(\rho)\right)^{-mt_2}\rho^{n-m}.
\end{align*}

In the same manner,  \eqref{39} with  $t_1=-\frac{(1+t_2)(m-1)m}{{m^*}}$ leads to
$$
\int_{B_{\rho/2}} \big(w^{-(1+t_2)(m-1)}|D\eta|\big)^m\; dx
\leq C\left(M(\rho)-M\left(\frac{\rho}{2}\right)+A(\rho)\right)^{(1+t_2)(m-1)m}\rho^{n-m}
$$
whence we have
$$
\int_{B_{\rho/2}} \eta^{m-1}{|Du|}^{m-1}|D\eta|\; dx 
\leq  C{\left(M(\rho)-M\left(\frac{\rho}{2}\right)+A(\rho)\right)}^{m-1}\rho^{n-m}.
$$

Further on, from the Young inequality and the definition of $\m$, we discover
\begin{align*}
S_2 &\leq CA(\rho)^{-1} \int_{B_{\rho/2}} \eta^m \left(1+\varphi^\frac{m}{m-1}\right)\;dx + CA(\rho)^{m-1} \int_{B_{\rho/2}} |D\eta|^m\;dx\\
&\leq C A(\rho)^{-1} \left(\int_{B_{\rho/2}} \eta^m \;d\m + \int_{B_{\rho/2}} \eta^m \;dx\right) + CA(\rho)^{m-1} \int_{B_{\rho/2}} |D\eta|^m\;dx\\
&\leq C A(\rho)^{-1}[S_3+S_4] + C A(\rho)^{m-1}  \rho^{n-m}.
\end{align*}
But since $\m(B_{\rho/2})\leq A(\rho)^m \rho^{\alpha}$ from \eqref{M_est2}, we have
$$
S_3 \leq C A(\rho)^m \rho^{\alpha} \leq C A(\rho)^m \rho^{n-m}
$$
because of $0<\rho<\mathrm{diam}\,\Omega$ and  $\alpha>n-m.$
Using the Poincar\'e inequality, we dicover
$$
S_4 \leq C \rho^m\int |D\eta|^m\; dx \leq C{A(\rho)}^m\rho^{n-m}
$$
whence 
$$
S_2 \leq C A(\rho)^{m-1}  \rho^{n-m}.
$$

Therefore, the bounds for $S_1$--$S_4$ and \eqref{40} yield
\begin{equation}\label{41}
\int_{B_{\rho/2}} \eta^m{|Du|}^m \;dx \leq  C\big(M(\rho)+A(\rho)\big){\left(M(\rho)-M\left(\frac{\rho}{2}\right)+A(\rho)\right)}^{m-1}\rho^{n-m}.
\end{equation}

On the other hand, we have
$$
\int_{B_{\rho/2}} {|D(\eta w^{-1})|}^m \leq  C \left(
\int_{B_{\rho/2}} w^{-m}|D\eta|^m\;dx + \int_{B_{\rho/2}} \eta^m{|Du|}^m \;dx \right)
$$
and, keeping in mind  $w^{-1}\leq M(\rho)+A(\rho),$ we apply \eqref{39} with $t_1=\frac{1-m}{{m^*}}$  to get
\begin{align*}
\int_{B_{\rho/2}} w^{-m}|D\eta|^m\;dx \leq&\ C \rho^{-m} \int_{B_{\rho/2}} w^{-1} w^{1-m}\;dx\\
\leq&\ 
C\big(M(\rho)+A(\rho)\big) \rho^{-m} \int_{B_{\rho/2}}  w^{1-m}\;dx\\
\leq&\ C\big(M(\rho)+A(\rho)\big){\left(M(\rho)-M\left(\frac{\rho}{2}\right)+A(\rho)\right)}^{m-1}\rho^{n-m}.
\end{align*}

The last bound, together with \eqref{41} gives
$$
\int_{B_{\rho/2}} {|D(\eta w^{-1})|}^m\;dx \leq  C\big(M(\rho)+A(\rho)\big){\left(M(\rho)-M\left(\frac{\rho}{2}\right)+A(\rho)\right)}^{m-1}\rho^{n-m},
$$
and this completes the proof of Lemma~\ref{lemM}.
\end{proof}

Once having the result of Lemma~\ref{lemM} it is easy to extend the H\"older continuity of the weak solutions up to the boundary of $\Omega,$ thanks of the $m$-thickness condition \eqref{3}.

\subsection{Proof of Theorem \ref{thm2}}
Let $x_0\in \partial\Omega$ be any point and set $B_\rho =B_\rho(x_0)$ for the sake of brevity.
Since $(m-1)p+\lambda>n$, $mq+\mu>n$ and $m+\omega>n$, there exist positive constants $\overline{\lambda}$, $\overline{\mu}$ and $\overline{\omega}$ such that $n<(m-1)p+\overline{\lambda}<(m-1)p+\lambda$, $n<mq+\overline{\mu}<mq+\mu$ and $n<m+\overline{\omega}<m+\omega$. It follows from Proposition~\ref{MSa} that $L^{p,\lambda}(B_\rho)\subset L^{p,\overline{\lambda}}(B_\rho)$, $L^{q,\mu}(B_\rho)\subset L^{q,\overline{\mu}}(B_\rho)$ and $\mathcal{R}^{1,\omega}(B_\rho)\subset \mathcal{R}^{1,\overline{\mu}}(B_\rho).$ In particular, since $\varphi\in L^{p,\lambda}$, $\psi\in L^{q,\mu}$ and $\nu\in \mathcal{R}^{1,\omega}$, we get from the definition
\begin{equation}\label{A-bar}
\begin{aligned}
{\|\varphi\|}_{L^{p,\overline{\lambda}}(B_\rho)}^p \leq	 {\|\varphi\|}_{L^{p,\lambda}(B_\rho)}^p\rho^{\lambda-\overline{\lambda}},
&\quad {\|\psi\|}_{L^{q,\overline{\mu}}(B_\rho)} \leq {\|\psi\|}_{L^{q,\mu}(B_\rho)}\rho^{\mu-\overline{\mu}} \\
  {\|\nu\|}_{\mathcal{R}^{1,\overline{\mu}}(B_\rho)} &\leq {\|\nu\|}_{\mathcal{R}^{1,\omega}(B_\rho)}\rho^{\omega-\overline{\omega}}.
\end{aligned}
\end{equation}
Indeed, for any $B_s(y)\subset\mr^n,$ we have
\begin{align*}
\frac{1}{s^{\overline{\lambda}}}\int_{B_\rho\cap B_s(y)} \varphi^p\; dx &= s^{\lambda-\overline{\lambda}} \frac{1}{s^{\lambda}}\int_{B_\rho\cap B_s(y)} \varphi^p\; dx \\
& \leq \rho^{\lambda-\overline{\lambda}} \frac{1}{s^{\lambda}}\int_{B_\rho\cap B_s(y)} \varphi^p\; dx 
\leq {\|\varphi\|}_{L^{p,\lambda}(B_\rho)}^p\rho^{\lambda-\overline{\lambda}}
\end{align*}
if $\rho \ge s$ and 
\begin{align*}
\frac{1}{s^{\overline{\lambda}}}\int_{B_\rho\cap B_s(y)} \varphi^p\; dx \leq&\ \frac{1}{\rho^{\overline{\lambda}}}\int_{B_\rho\cap B_s(y)} \varphi^p\; dx
\leq \rho^{\lambda-\overline{\lambda}} \frac{1}{\rho^{\lambda}}\int_{B_\rho} \varphi^p\; dx \\
\leq&\ {\|\varphi\|}_{L^{p,\lambda}(B_\rho)}^p\rho^{\lambda-\overline{\lambda}}
\end{align*}
if $\rho \leq s$ and we can derive the second and third inequalities in \eqref{A-bar} in an analoguous way.

Take now a cut-off function $\eta\in C^\infty_0\left(B_{\rho/2}\right)$ so that $0\leq\eta\leq1,$ $\eta=1$ on $B_{\rho/4}$ and $|D\eta|\leq C/\rho.$ It follows from the
 $m$-thickness conditions \eqref{3} that
$$
\mathrm{Cap}_m\big(B_{\rho/4}\setminus\Omega,B_{\rho/2}\big)\geq A_\Omega\,
\mathrm{Cap}_m\big(B_{\rho/4},B_{\rho/2}\big)
= C\rho^{n-m}\quad \forall \rho\leq r_0.
$$
On the other hand,
$$
B_{\rho/4}\setminus\Omega \subset \left\{x\in B_{\rho/4}\colon\ \tilde{u}(x)=0\right\}
$$
and therefore
$$
\mathrm{Cap}_m\Big(\left\{x\in B_{\rho/4}\colon\ \tilde{u}(x)=0\right\},B_{\rho/2}\Big)\geq 
\mathrm{Cap}_m\big(B_{\rho/4}\setminus\Omega,B_{\rho/2}\big).
$$
We have $\eta w^{-1}=M(\rho)+\overline{A}(\rho)$ on $\left\{x\in B_{\rho/4}\colon\ \tilde{u}(x)=0\right\}$ with
$$
\overline{A}(\rho) = \rho + {\|\varphi\|}_{L^{p,\overline{\lambda}}(B_\rho)}^\frac{1}{m-1} + {\|\psi\|}_{L^{q,\overline{\mu}}(B_\rho)}^\frac{1}{m} + {\|\nu\|}_{\mathcal{R}^{1,\overline{\mu}}(B_\rho)}^\frac{1}{m}
$$
and thus
$$
\int_{B_{\rho/2}} \left|
D\left(\frac{\eta w^{-1}}{M(\rho)+\overline{A}(\rho)}
\right)\right|^m\;dx \geq
\mathrm{Cap}_m\Big(\left\{x\in B_{\rho/4}\colon\ \tilde{u}(x)=0\right\},B_{\rho/2}\Big).
$$

Putting together all these inequalities, 
Lemma~\ref{lemM} gives
\begin{align*}
\rho^{n-m} \leq&\ C\, \mathrm{Cap}_m\Big(\left\{x\in B_{\rho/4}\colon\ \tilde{u}(x)=0\right\},B_{\rho/2}\Big) \\ 
\leq&\ C{\left((M(\rho)+\overline{A}(\rho)\right)}^{-m}\int_{B_{\rho/2}} {|D(\eta w^{-1})|}^m\\ 
\leq&\  C{\left(M(\rho)+\overline{A}(\rho)\right)}^{1-m}{\left(M(\rho)-M\left(\frac{\rho}{2}\right)+\overline{A}(\rho)\right)}^{m-1}\rho^{n-m}.
\end{align*}
Thus, we find 
$$
M\left(\frac{\rho}{2}\right) \leq \frac{C-1}{C}\left(M(\rho)+\overline{A}(\rho)\right)
$$
for all $\rho\leq R$ where $R$ depends on $r_0$ from \eqref{3},
and it follows from Proposition~\ref{Pinterpolation} that 
$$
M(\rho) \leq C \left(\left(\frac{\rho}{R}\right)^{\alpha'} M(R) + \overline{A}(\rho^\tau R^{1-\tau})\right)
$$
for any $0<\tau<1$ and $\rho\leq R$ with an exponent $\alpha'>0.$ Since 
$$
\overline{A}(\rho^\tau R^{1-\tau})\leq C(\rho^\tau R^{1-\tau})^{\alpha''}
$$ 
with $\alpha'' = \min\{\frac{\lambda-\overline{\lambda}}{(m-1)p},\frac{\mu-\overline{\mu}}{mq},\frac{\omega-\overline{\omega}}{m},1\}$ as it follows from \eqref{A-bar}, we have
$$
M(\rho) \leq C \rho ^ \alpha,
$$
where $\alpha = \min\{\alpha',\tau\alpha''\}.$

Repeating the above procedure with $-u(x)$ instead of $u(x),$ we get finally
\begin{equation}\label{boundary_oscillation}
\sup_{B_\rho(x_0)} |u| \leq C\rho^\alpha
\end{equation}
for all $x_0\in\partial\Omega$ and all $\rho\in(0,R).$

\bigskip

With Corollary~\ref{crl1} and \eqref{boundary_oscillation} at hand, it is standard matter to get  H\"older continuity up to the boundary as claimed in Theorem~\ref{thm2}. For, we will distinguish between various cases for arbitrary two points $x,y\in\overline{\Omega}.$ 

\textit{Case 1: $\mathrm{dist\,}(x,\partial\Omega)\geq R/2$ and $\mathrm{dist\,}(y,\partial\Omega)\geq R/2.$} Then 
$$
\frac{|u(x)-u(y)|}{|x-y|^\alpha} \leq H
$$
as it follows from Corollary~\ref{crl1}. 

\textit{Case 2: $0<\mathrm{dist\,}(x,\partial\Omega),\ \mathrm{dist\,}(y,\partial\Omega)<R/2.$} Let $\delta=\mathrm{dist\,}(y,\partial\Omega).$ We have
$$
\frac{|u(x)-u(y)|}{|x-y|^\alpha} \leq \frac{\osc_{B_{|x-y|}(y)} u}{|x-y|^{\alpha}},
$$
while
$$
\osc_{B_{|x-y|}(y)} u \leq C |x-y|^\alpha \left(
\delta^{-\alpha}\osc_{B_\delta(y)} u+1\right)
$$
as consequence of \cite[Theorem~4.1]{L}. This way,
$$
\frac{|u(x)-u(y)|}{|x-y|^\alpha} \leq
C  \left(
\delta^{-\alpha}\osc_{B_\delta(y)} u+1\right)
$$
for each $x\in B_\delta(y)$ with a suitable exponent $\alpha\in(0,1).$ 
Pick now a point $y_0\in\partial\Omega$ with the property $|y_0-y|=\mathrm{dist\,}(y,\partial\Omega).$ Since $B_\delta(y)\subset B_{2\delta}(y_0),$ we have from \eqref{boundary_oscillation}
that
$$
\osc_{B_\delta(y)} u \leq \osc_{B_{2\delta}(y_0)}u \leq 2 \sup_{B_{2\delta}(y_0)} |u| \leq C\delta^\alpha
$$
whence
$$
\frac{|u(x)-u(y)|}{|x-y|^\alpha} \leq H
$$
for all $x\in B_\delta(y)$ and all $y\in\Omega$  with $\mathrm{dist\,}(y,\partial\Omega)< R/2.$ Further, if $|x-y|\geq\delta$
and $\mathrm{dist\,}(x,\partial\Omega)< R/2,$
 take a point $x_0\in\partial\Omega$ with the property $|x_0-x|=\mathrm{dist\,}(x,\partial\Omega).$ Since
$$
|x-x_0| \leq |x-y_0| \leq |x-y|+|y-y_0| = |x-y|+ \delta \leq 2|x-y|
$$
and $u(x_0)=u(y_0)=0,$
we have
\begin{align*}
|u(x)-u(y)| \leq&\ |u(x)-u(x_0)| + |u(x_0)-u(y_0)| + |u(y_0)-u(y)| \\
\leq&\ C \big(|x-x_0|^\alpha + |y-y_0|^\alpha \big) \\
\leq&\ C|x-y|^\alpha,
\end{align*}
whence
$$
\frac{|u(x)-u(y)|}{|x-y|^\alpha} \leq H
$$
for all $x,y\in\Omega$ with $\mathrm{dist\,}(x,\partial\Omega), \mathrm{dist\,}(y,\partial\Omega)\in(0, R/2)$ and such that $|x-y|\geq\delta.$

\textit{Case 3: $\mathrm{dist\,}(x,\partial\Omega)\geq R/2$ and $0<\mathrm{dist\,}(y,\partial\Omega)< R/2.$}  It suffices to take a point $z$ lying on the segment with end $x$ and $y$ and such that $\mathrm{dist\,}(z,\partial\Omega)=R/2$ to get
$$
\frac{|u(x)-u(y)|}{|x-y|^\alpha} \leq 
\frac{|u(x)-u(z)|}{|x-z|^\alpha} +
\frac{|u(z)-u(y)|}{|z-y|^\alpha}.
$$
Thus, the desired estimate reduces to the cases already considered.

\textit{Case 4: $y\in\partial\Omega.$}
It follows from \eqref{boundary_oscillation} that
$$
\frac{|u(x)-u(y)|}{|x-y|^\alpha} \leq H
$$
for all $x\in\overline\Omega$ such that $|x-y|<R,$ while
$$
\frac{|u(x)-u(y)|}{|x-y|^\alpha} \leq 2 \|u\|_{L^\infty(\Omega)} R^{-\alpha} \leq H
$$
if $|x-y|\geq R,$ as consequence of \eqref{10}.

It remains to take the smallest of the exponents $\alpha$ in the above considerations to complete the proof of Theorem~\ref{thm2}.\hfill $\qed$

\subsection{H\"older continuity under natural structure conditions}

Theorem~\ref{thm2} asserts global H\"older continuity of the weak solutions to the Dirichlet problem \eqref{1} under
the same hypotheses which ensure global boundedness of the solutions. However,  it happens very often that one already disposes of an \textit{a~priori} bound for $\|u\|_{L^\infty(\Omega)}$ as consequence, for example, of strong monotonicity of the principal part $\ba(x,u,Du)$ with respect to $Du,$ or sign condition on $u.b(x,u,Du)$ (see e.g. \cite{JMAA,NoDEA} and the references therein), etc. What is the natural question to arise in this situation is whether the \textit{bounded} weak solutions to \eqref{1} remain globally H\"older continuous in ${\Omega}$ if the $|\xi|^{m\left(1-\frac{1}{{m^*}}\right)}$-growth of $b(x,z,\xi)$ in \eqref{4} is relaxed to $|\xi|^m.$

More precisely, let us weaken the \textit{controlled}  growth assumptions \eqref{4} to the \textit{natural structure} conditions of Ladyzhenskaya and Ural'tseva. In other words, let $\varphi\in L^{p,\lambda}(\Omega)$ with $p>\frac{m}{m-1},$ $\lambda\in(0,n)$ and $(m-1)p+\lambda>n;$
$\psi\in L^{q,\mu}(\Omega)$ with $q\geq 1,$ $\mu\in(0,n)$ and $mq+\mu>n,$ and suppose there exist 
a non-decreasing function $\Lambda(t)$ and a non-increasing function $\gamma(t),$ both positive and continuous, such that 
\begin{equation}\label{4natural}
\begin{cases}
|\ba(x,z,\xi)| \leq \Lambda(|z|)\left(\varphi(x)+|\xi|^{m-1}\right),\\[4pt]
|b(x,z,\xi)| \leq \Lambda(|z|)\left(\psi(x)+|\xi|^{m}\right)
\end{cases}
\end{equation}
and
\begin{equation}\label{5natural}
\ba(x,z,\xi)\cdot\xi\geq\gamma(|z|)|\xi|^m-\Lambda(|z|)\varphi(x)^\frac{m}{m-1}
\end{equation}
for a.a. $x\in\Omega$ and all $(z,\xi)\in\mr\times\mr^n.$ 

Indeed, a \textit{bounded weak solution} to \eqref{1} is
a function $u\in L^\infty(\Omega)\cap W^{1,m}_0(\Omega)$ such that
$$
\int_\Omega \ba(x,u,Du)\cdot Dv\; dx + \int_\Omega b(x,u,Du)v\; dx + \int_\Omega v\; d\nu = 0\quad \forall v \in L^\infty(\Omega)\cap  W^{1,m}_0(\Omega).
$$

It is worth noting that in the proof of Corollary~\ref{crl1} above, we reduced \eqref{4} and \eqref{5} just to  \eqref{4natural} and \eqref{5natural}, respectively. Further, it is easy to check that the result of Lemma~\ref{lemM} remains valid for \textit{bounded} weak solutions to \eqref{1} if \eqref{4natural} and \eqref{5natural} are required instead of \eqref{4} and \eqref{5}. This way, we have
\begin{theorem}\label{thm3}
Under the hypotheses \eqref{2}, \eqref{2-1},
 \eqref{3}, \eqref{4natural} and \eqref{5natural}, each {\em bounded} weak solution of the Dirichlet problem \eqref{1} is  H\"older continuous in $\overline\Omega$ with H\"older exponent and constant depending on the data of \eqref{1} and
on $\|u\|_{L^\infty(\Omega)}.$  
\end{theorem}

\end{document}